\documentclass[11pt]{article}
\usepackage{amsfonts,amsmath,amssymb,amsthm}
\usepackage{bm}
\usepackage{graphicx} 
\usepackage[margin=1in]{geometry}
\usepackage{enumerate}
\usepackage{tikz-cd} 

\newtheorem{theorem}{Theorem}
\newtheorem{proposition}{Proposition}
\newtheorem{corollary}{Corollary}
\newtheorem{lemma}{Lemma}
\newtheorem{remark}{Remark}
\newtheorem{definition}{Definition}
\newtheorem{assumption}{Assumption}

\DeclareMathOperator{\curl}{curl}
\DeclareMathOperator{\divv}{div}

\newcommand{\RT}{\mathcal{RT}}
\newcommand{\Nc}{\mathcal{N}}

\usepackage[textsize=footnotesize]{todonotes}
\title{Well-Posedness and Approximation of Weak Solutions to Time Dependent Maxwell's Equations with $L^2$-Data
\thanks{Dedication: This work is dedicated to the memory of our teacher and mentor Prof. Dr. Ronald H.W. Hoppe.}
\thanks{This work is partially supported by the Office of Naval Research (ONR) under Award NO: N00014-24-
1-2147 and the Air Force Office of Scientific Research (AFOSR) under Award NO: FA9550-25-1-0231.}}
\author{Harbir Antil \\ Center for Mathematics and Artificial Intelligence and Department of Mathematical Sciences \\ George Mason University, Fairfax, VA 22030}
\date{}

\begin{document}
\maketitle

\begin{abstract}
We study Maxwell's equations in conducting media with perfectly conducting boundary conditions on Lipschitz domains, allowing rough material coefficients and $L^2$-data. Our first contribution is a direct proof of well-posedness of the first-order weak formulation, including solution existence and uniqueness, an energy identity, and continuous dependence on the data. The argument uses interior-in-time mollification to show uniqueness while avoiding reflection techniques. Existence is via the well-known Galerkin method (cf.~Duvaut and Lions \cite[Eqns.~(4.31)--(4.32), p.~346; Thm.~4.1]{GDuvaut_JLLions_1976a}). For completeness, and to make the paper self-contained, a complete proof has been provided. 

Our second contribution is a structure-preserving semi-discrete finite element method based on the Nédélec/Raviart--Thomas de Rham complex. The scheme preserves a discrete Gauss law for all times and satisfies a continuous-in-time energy identity with stability for nonnegative conductivity. With a divergence-free initialization of the magnetic field (via potential reconstruction or constrained $L^2$ projection), we prove convergence of the semi-discrete solutions to the unique weak solution as the mesh is refined. The analysis mostly relies on projector consistency, weak-* compactness in time-bounded $L^2$ spaces, and identification of time derivatives in dual spaces.
\end{abstract}

\section{Introduction}

\paragraph{Problem setting.}
Let $\Omega\subset\mathbb R^3$ be a bounded Lipschitz domain and $T>0$. Material parameters satisfy
\[
\varepsilon\in L^\infty(\Omega;\mathbb{R}^{3\times3}),\quad
\mu\in L^\infty(\Omega;\mathbb{R}^{3\times3}),\quad
\sigma\in L^\infty(\Omega;\mathbb{R}^{3\times3}),
\]
all symmetric, with $\varepsilon$ and $\mu$ uniformly elliptic and $\sigma$ nonnegative. Given
$f\in L^2(0,T;L^2(\Omega)^3)$ and $E_0,B_0\in L^2(\Omega)^3$ with $\divv B_0=0$ in $L^2(\Omega)$, we consider
\[
\begin{cases}
\varepsilon\,\partial_t E+\sigma E-\curl(\mu^{-1}B)=f,\\[0.5mm]
\partial_t B+\curl E=0,
\end{cases}
\quad
E\times n=0 \ \text{ on }\partial\Omega\times(0,T),\qquad
E(0)=E_0,\ \ B(0)=B_0.
\]
Our weak solution notion (Definition~\ref{def:weak_soln}) requires
\[
E\in H^1\!\big(0,T;H_0(\curl;\Omega)^*\big)\cap C([0,T];L^2(\Omega)^3),\qquad
B\in H^1\!\big(0,T;H(\curl;\Omega)^*\big)\cap C([0,T];L^2(\Omega)^3),
\]
satisfying the usual variational identities against $\psi\in H_0(\curl;\Omega)$ and $\phi\in H(\curl;\Omega)$ for a.e.\ $t\in(0,T)$.

\paragraph{Scope and goals.}
The primary goal of this paper is to prove convergence of a conforming finite element
\emph{semi-discretization in space} (first–order $H(\curl)$/$H(\divv)$ formulation) to the continuous weak solution for data
$f \in L^2(0,T;L^2(\Omega)^3)$ and $E_0,B_0 \in L^2(\Omega)^3$. To the best of our knowledge, this result is new. 
A byproduct of the arguments needed for the convergence analysis is a direct, self-contained proof of well-posedness for the continuous problem with the above stated data regularity. 

\paragraph{Positioning within the literature.}
There are comparatively few references that treat well-posedness for Maxwell's equations at this level of generality. The classical monograph of Duvaut and Lions \cite[Eqns.~(4.31)--(4.32), p.~346; Thm.~4.1]{GDuvaut_JLLions_1976a} establishes existence and uniqueness of $L^\infty(0,T;L^2(\Omega)^3)$ solutions. Their concise presentation makes it nontrivial to infer continuity in time, continuous dependence on data, and additional time-regularity. 
Notice that such results can also be inferred from \cite[Thm.~2.4]{AKirsch_ARieder_2016a} and \cite[Lemma~3.2]{IYousept_2020a} via the semigroup/mild-solution framework. The latter result uses the abstract theory of $C_0$-semigroups (cf.\ Ball \cite{JMBall_1977a}). We also refer to \cite[Sec.~8.2]{RLeis_1997a} for a discussion on well-posedness using the spectral theorem. See also \cite[Sec.~7.8]{MFabrizio_MMorro_2003a} where the authors consider weak solutions satisfying the free charge density law and Gauss law in the weak sense. Their well-posedness result uses density based arguments. 

The present paper gives a complete self-contained proof of well-posedness for the weak formulation of Maxwell's equations. 
Our approach may be viewed as an extension and clarification of the arguments in \cite{GDuvaut_JLLions_1976a}: (i) we provide complete details in $L^2$-data setting; (ii) uniqueness is obtained without resorting to time-reflection, using instead an interior-in-time mollification argument; and (iii) we establish the stated time-regularity and continuous dependence without auxiliary smoothing assumptions on the data. We collect these results here both for completeness and because several steps in the continuous analysis feed directly into the convergence proof of our numerical scheme. This synthesis honors Ronald H.W. Hoppe, who made influential contributions to computational electromagnetics \cite{RHWHoppe_1982a}.

On the numerical side, early finite element discretizations for Maxwell's equations were developed in the semi-discrete (space-only) setting in \cite{PMonk_1992a} and in fully discrete form in \cite{PJCiarlet_JZou_1999a}, typically for second–order (in space/time) formulations. See also \cite{CGMakridakis_PMonk_1995} for analysis and convergence estimates of a scheme closely related to ours (N\'ed\'elec and Raviart-Thomas (RT) discretization for $E$ and $B$); \cite{JLi_2011a} for a fully time-discrete leapfrog analysis; and \cite{JLi_2007a} for a piecewise-constant (in space) approximation of the electric field.

Recent work has addressed nonlinear and nonsmooth models. For Maxwell variational inequalities (MVIs) of the second kind in type-II superconductivity, \cite{MWinckler_IYousept_2019a} uses N\'ed\'elec elements for \(E\), piecewise constants for \(B\), and implicit Euler in time. For MVIs of the first kind in electric shielding, \cite{MHensel_IYousept_2022a} employs piecewise constants for \(E\) and N\'ed\'elec for \(H\) (with \(B=\mu H\)). Both assume sources in \(W^{1,\infty}(0,T;L^2)\). The quasi-MVI study \cite{MHensel_MWinckler_IYousept_2024a} treats both implicit Euler and leapfrog: N\'ed\'elec for \(E\) and piecewise constants for \(H\) in the former; piecewise constants for \(E\) and N\'ed\'elec for \(H\) in the latter. Remarkably for leapfrog scheme, the authors can handle sources which are of bounded variation type. 

To our knowledge, none of these works consider the specific conforming pairing analyzed here---N\'ed\'elec/RT in the first-order \(H(\mathrm{curl})/H(\mathrm{div})\) framework---together with \emph{minimal data regularity} \(f\in L^2(0,T;L^2)\), \(E_0,B_0\in L^2\), and prove convergence of the \emph{semi-discrete} solution to the weak continuous solution.

\paragraph{Contributions.}
\begin{itemize}
  \item A direct, proof of existence, uniqueness, and continuous dependence for the weak Maxwell system in the minimal-regularity class
  \[
  E\in H^1(0,T;H_0(\mathrm{curl};\Omega)^*) \cap C([0,T];L^2),\qquad
  B\in H^1(0,T;H(\mathrm{curl};\Omega)^*) \cap C([0,T]; L^2).
  \]
  \item A structure-preserving \emph{semi-discrete} FE formulation based on the N\'ed\'elec/RT de~Rham complex that (i) preserves a discrete Gauss law \emph{for all times} and (ii) satisfies a continuous-in-time discrete energy identity and stability for $\sigma\ge0$.
  \item Convergence of the semi-discrete solutions $(E_h,B_h)$ to the continuous unique weak solution as $h\to0$ under minimal data assumptions
  ($f\in L^2(0,T;L^2(\Omega)^3)$, $E_0,B_0\in L^2(\Omega)^3$), using mainly spatial projector consistency and weak-* compactness in $L^\infty(0,T;L^2)$.
\end{itemize}

\paragraph{Organization.}
Section~\ref{s:not} states the functional setting and weak formulation. Section~\ref{s:wellposed} proves uniqueness via time mollification and test-side smoothing and existence using the Galerkin method. Section~\ref{sec:semidiscrete} introduces the N\'ed\'elec/RT semi-discrete scheme, proves stability and discrete Gauss law preservation, and establishes convergence. Auxiliary results are collected in Appendix~\ref{app:appendix}.

\section{Notation and Preliminaries}\label{s:not}
Let \( \Omega \subset \mathbb{R}^3 \) be a bounded Lipschitz domain and \( 0 < T < +\infty \). Consider the time-dependent Maxwell system with anisotropic tensor coefficients:
\begin{subequations}\label{eq:Maxwell_strong}
\begin{align}
\varepsilon(x) \partial_t E - \mathrm{curl} \, (\mu^{-1}(x) B) + \sigma(x) E &= f(x,t), 
&& \mbox{in } \Omega \times (0,T) \label{eq:max1} \\
\partial_t B + \mathrm{curl} \, E &= 0, 
&& \mbox{in } \Omega \times (0,T) 
\label{eq:max2} 
\end{align}
\end{subequations}
subject to the initial and boundary conditions:
\begin{align*}
E(0,x) &= E_0(x), \quad B(0,x) = B_0(x), && \text{in } \Omega \\
E \times \nu &= 0  && \text{on } \partial \Omega \times (0,T)  
\end{align*}
Notice that \eqref{eq:max1} and \eqref{eq:max2} are the Amp\'ere-Maxwell
and Faraday laws, respectively. 

\begin{assumption}
\label{ass:data}
Throughout the article, the following conditions are assumed on the data:
\begin{itemize}
  \item \( \varepsilon(x),\sigma(x), \mu(x) \in L^\infty(\Omega; \mathbb{R}^{3 \times 3}) \) are symmetric tensor fields;
  \item There exist constants \( \varepsilon_0, \mu_0 \) such that
  \[
    \xi^\top \varepsilon(x) \xi \ge \varepsilon_0 |\xi|^2,\quad \xi^\top \mu(x) \xi \ge \mu_0 |\xi|^2,    
    \quad \text{for all } \xi \in \mathbb{R}^3,
    \ \text{a.e. } x \in \Omega;
  \]
  \item \(\sigma\) is symmetric positive semi-definite, \( f \in L^2(0,T; L^2(\Omega)^3) \), and \( E_0, B_0 \in L^2(\Omega)^3 \) with $\mathrm{div}\, B_0 = 0$ in $L^2(\Omega)$. 
\end{itemize}
\end{assumption}

Throughout, we use $\|\cdot\|_{L^2(\Omega)}$ to denote the $L^2$-norm and $(\cdot,\cdot)$ to denote the $L^2$-scalar product. We denote the space for vector valued functions by $L^2(\Omega)^3$, but will also interchangeably use $L^2$ (to minimize the notation) when it is clear from the context. 
For a given Banach space $X$, we denote its topological dual by $X^*$, moreover, $\langle\cdot,\cdot\rangle_X$ denotes the duality pairing between $X^*$ and $X$. 
We define the following Sobolev spaces 
\[
\begin{aligned}
H(\text{div}; \Omega) &:= 
\left\{ v \in L^2(\Omega)^3 : \mbox{div } v \in L^2(\Omega) \right\}, \\
H_0(\text{div}; \Omega) &:= 
\left\{ v \in H(\text{div}; \Omega) : \gamma_\nu(v) := \gamma v \cdot \nu = 0 \mbox{ on } \partial\Omega \right\}, \\
H(\text{div}^0; \Omega) 
&:= \left\{ v \in H(\text{div}; \Omega) : \mbox{div } v = 0 \text{ a.e. in }  \Omega \right\} , \\
H(\text{curl}; \Omega) 
&:= \left\{ v \in L^2(\Omega)^3 : \mathrm{curl}\, v \in L^2(\Omega)^3 \right\} ,
\\ 
H_0(\text{curl}; \Omega) &:= \left\{ v \in H(\text{curl}; \Omega) : \gamma_\tau(v) :=  \gamma v \times \nu = 0 \text{ on } \partial \Omega \right\} .
\end{aligned}
\]
Here $\gamma_\nu$ is the normal trace \cite{VGirault_PARaviart_1986a} and $\gamma_\tau$ the tangential trace \cite{ABuffa_PJCiarlet_2001a,MON03}.

\begin{definition}[Weak solution to Maxwell's Equations]
\label{def:weak_soln}
The electric and magnetic fields $(E,B)$ solve \eqref{eq:Maxwell_strong} weakly if and only if 
\[
\begin{aligned}
E &\in H^1(0,T; H_0(\mathrm{curl};\Omega)^*) \cap L^\infty(0,T;L^2(\Omega)^3), \\
B &\in H^1(0,T; H(\curl;\Omega)^*) \cap L^\infty(0,T;L^2(\Omega)^3) ,
\end{aligned}
\]
and satisfy, for almost every \( t \in (0,T) \), the variational formulation: 
\begin{subequations}\label{eq:MaxWeak}
\begin{align}
\langle \varepsilon \partial_t E, \psi \rangle_{H_0(\mathrm{curl}; \Omega)} + (\sigma E, \psi) - (\mu^{-1}B, \mathrm{curl}\, \psi) &= (f, \psi), 
&&\forall \psi \in H_0(\mathrm{curl}; \Omega), \\
\langle \partial_t B, \phi \rangle_{H(\curl; \Omega)} + ( E, \curl \phi ) &= 0, 
&&\forall \phi \in H(\curl; \Omega),
\end{align}
\end{subequations}
with initial data \( E(0) = E_0 \in L^2(\Omega)^3 \), \( B(0) = B_0 \in L^2(\Omega)^3 \). Moreover, if $\divv B_0 = 0$ in $L^2(\Omega)$, then additionally, $B \in L^\infty(0,T;H(\divv^0;\Omega))$.
\end{definition}
For the justification of pointwise evaluation of the initial condition in time under the regularity given in Definition~\ref{def:weak_soln}, we refer to \cite[Proposition 2.19]{AKaltenbach_2023a}. Notice that in Corollary~\ref{cor:C0-in-time} we establish that $(E,B)$ solving \eqref{eq:MaxWeak} also fulfills $E \in C([0,T];L^2(\Omega)^3)$ and $B \in C([0,T];L^2(\Omega)^3)$.

\section{Well Posedness of Maxwell's Equations}
\label{s:wellposed}

This section is organized as follows. First in Theorem~\ref{thm:uniq-cont} we establish that \eqref{eq:MaxWeak} has a unique solution. This is tricky because we cannot use $E$ and $B$ as test functions to use the standard energy argument to establish uniqueness. Instead we develop a mollification in time argument. Corollary~\ref{cor:C0-in-time} shows that $E$ and $B$ are in $C([0,T];L^2(\Omega)^3)$. Next, in Theorem~\ref{thm:main} we show existence of solution to \eqref{thm:uniq-cont} via a Galerkin type argument. Proposition~\ref{prop:divB} shows that $B$ is solenoidal, i.e., $\divv B = 0$ in a certain sense.

The following result will be helpful in showing uniqueness of solution:
\begin{lemma}[Characterization of $H(\curl)$ via distributional curl]
\label{lem:Hcurl-char}
Let $v\in L^2(\Omega)^3$. If there exists $g\in L^2(\Omega)^3$ such that
\[
(v,\ \curl \phi)_{L^2(\Omega)^3} \;=\; (g,\ \phi)_{L^2(\Omega)^3}\qquad
\forall\,\phi\in C_c^\infty(\Omega)^3,
\]
then $v\in H(\curl;\Omega)$ and $\curl v=g$ in $L^2(\Omega)^3$.
\end{lemma}
\begin{proof}
Define a linear functional $\Lambda$ on $C_c^\infty(\Omega)^3$ by
$\Lambda(\phi):=(v,\curl\phi)_{L^2}$. The hypothesis yields
$\Lambda(\phi)=(g,\phi)_{L^2}$, hence
$|\Lambda(\phi)|\le \|g\|_{L^2}\,\|\phi\|_{L^2}$ for all $\phi\in C_c^\infty$.
Thus $\Lambda$ extends continuously (by density) to $L^2(\Omega)^3$ and the
Riesz representative of $\Lambda$ is $g$. By the definition of the
distributional curl, this precisely means that $\curl v=g$ as an $L^2$–field;
hence $v\in H(\curl;\Omega)$ with $\curl v=g$ in $L^2(\Omega)^3$.
\end{proof}

\begin{theorem}[Solution to \eqref{eq:MaxWeak} is unique]\label{thm:uniq-cont}
Let $(E_i,B_i)$, $i=1,2$, be weak solutions to \eqref{eq:MaxWeak} with the same data 
$f,E_0,B_0$ in the sense of Definition~\ref{def:weak_soln}. Then $E_1\equiv E_2$ and $B_1\equiv B_2$ on $(0,T)$.
\end{theorem}

\begin{proof}
Let $(\tilde E,\tilde B):=(E_1-E_2,B_1-B_2)$. Then $(\tilde E,\tilde B)$ solves the homogeneous system
\begin{equation}\label{eq:weak-diff}
\begin{aligned}
\langle \varepsilon \partial_t \tilde E, \psi \rangle_{H_0(\curl;\Omega)} 
+ (\sigma \tilde E, \psi) - (\mu^{-1} \tilde B, \curl \psi) &= 0
&& \forall \psi\in H_0(\curl;\Omega), \\
\langle \partial_t \tilde B, \phi \rangle_{H(\curl;\Omega)} 
+ ( \tilde E, \curl \phi ) &= 0
&& \forall \phi\in H(\curl;\Omega),
\end{aligned}
\end{equation}
with $\tilde E(0)=0$ and $\tilde B(0)=0$ in $L^2(\Omega)^3$.

\medskip\noindent
{\bf Step 1 (interior time–mollification).}
Let $\rho\in C_c^\infty(\mathbb R)$ be even, nonnegative, with $\int_{\mathbb R}\rho=1$ and
$\operatorname{supp}\rho\subset[-1,1]$. For $\delta\in(0,T/2)$ set
$\rho_\delta(s):=\delta^{-1}\rho(s/\delta)$, so $\operatorname{supp}\rho_\delta\subset[-\delta,\delta]$.

For $t\in(\delta,T-\delta)$ define the (interior) mollifications
\[
\tilde E^\delta(t):=\int_{\mathbb R}\rho_\delta(s)\,\tilde E(t-s)\,ds
=\int_{-\delta}^{\delta}\rho_\delta(s)\,\tilde E(t-s)\,ds,\qquad
\tilde B^\delta(t):=\int_{-\delta}^{\delta}\rho_\delta(s)\,\tilde B(t-s)\,ds.
\]
This is well-defined because, for any $t\in(\delta,T-\delta)$ and any $s$ with
$\rho_\delta(s)\neq0$ (hence $|s|\le\delta$),
\begin{equation}\label{eq:window}
t-s\ \ge\ t-\delta\ >\ 0,\qquad
t-s\ \le\ t+\delta\ <\ T,
\end{equation}
so $t-s\in(0,T)$ and only values of $(\tilde E,\tilde B)$ inside their domain are sampled.
Standard properties of convolution yield
\[
\tilde E^\delta,\tilde B^\delta\in C^\infty\big((\delta,T-\delta);L^2(\Omega)^3\big),\qquad
\partial_t \tilde E^\delta=\int_{-\delta}^{\delta}\rho_\delta'(s)\,\tilde E(t-s)\,ds,\quad
\partial_t \tilde B^\delta=\int_{-\delta}^{\delta}\rho_\delta'(s)\,\tilde B(t-s)\,ds,
\]
so $\partial_t\tilde E^\delta\in C^\infty\big((\delta,T-\delta);L^2(\Omega)^3\big)$ and likewise for
$\partial_t\tilde B^\delta$.

Moreover, for every compact interval $J=[a,b]\Subset(0,T)$ and $\delta<\min\{a,T-b\}$,
\[
\tilde E^\delta\to \tilde E\ \ \text{in }L^2\big(J;L^2(\Omega)^3\big),\qquad
\tilde B^\delta\to \tilde B\ \ \text{in }L^2\big(J;L^2(\Omega)^3\big),
\]
and the mollifications are $L^\infty$–stable:
\begin{equation}\label{eq:Linfstable}
\|\tilde E^\delta\|_{L^\infty((\delta,T-\delta);L^2)}\le \|\tilde E\|_{L^\infty((0,T);L^2)},\qquad
\|\tilde B^\delta\|_{L^\infty((\delta,T-\delta);L^2)}\le \|\tilde B\|_{L^\infty((0,T);L^2)}.
\end{equation}

\medskip\noindent
\textbf{Step 2 (regularized Maxwell system on $(\delta,T-\delta)$).}
Fix $t\in(\delta,T-\delta)$, $\psi\in H_0(\curl;\Omega)$, and $\phi\in H(\curl;\Omega)$. Using
\eqref{eq:weak-diff} at times $t-s$ and Fubini,
\[
\begin{aligned}
\langle \varepsilon \partial_t \tilde E^\delta(t),\psi\rangle_{H_0(\curl;\Omega)}
&=\int\rho_\delta(s)\,\langle \varepsilon \partial_t \tilde E(t-s),\psi\rangle\,ds
=(\mu^{-1}\tilde B^\delta(t),\curl\psi)-(\sigma \tilde E^\delta(t),\psi),\\
\langle \partial_t \tilde B^\delta(t),\phi\rangle_{H(\curl;\Omega)}
&=\int\rho_\delta(s)\,\langle \partial_t \tilde B(t-s),\phi\rangle\,ds
=-( \tilde E^\delta(t),\, \curl \phi ) . 
\end{aligned}
\]
Thus, for a.e.\ $t\in(\delta,T-\delta)$,
\begin{equation}\label{eq:reg-eqs}
\begin{aligned}
\langle \varepsilon \partial_t \tilde E^\delta(t),\psi\rangle_{H_0(\curl;\Omega)} + (\sigma \tilde E^\delta(t),\psi) - (\mu^{-1}\tilde B^\delta(t),\curl\psi) &= 0,\\
\langle \partial_t \tilde B^\delta(t),\phi\rangle_{H(\curl;\Omega)} + (\tilde E^\delta(t), \curl\phi) &= 0.
\end{aligned}
\end{equation}
Because the time mollification yields $\partial_t \tilde B^\delta(t)\in L^2(\Omega)^3$ for a.e.\ $t$, in \eqref{eq:reg-eqs}$_2$ we have 
\[
\big\langle \partial_t \tilde B^\delta(t),\ \phi\big\rangle_{H(\curl;\Omega)} =(\partial_t \tilde B^\delta(t),\ \phi)_{L^2}.
\]
Hence \eqref{eq:reg-eqs}$_2$ reads 
\begin{equation}\label{eq:aux2}
\big(\tilde E^\delta(t),\ \curl\phi\big)_{L^2} \;=\; -\,(\partial_t \tilde B^\delta(t),\ \phi)_{L^2}
\qquad\forall\,\phi\in H(\curl;\Omega).
\end{equation}
In particular, \eqref{eq:aux2} holds for $\phi \in C_c^\infty(\Omega)^3 \subset H(\curl;\Omega)$. Then 
Lemma~\ref{lem:Hcurl-char} with $v = \tilde{E}^\delta(t)$ and $g = -\partial_t \tilde{B}^\delta(t)$ implies 
\begin{equation}\label{eq:baba}
\tilde{E}^\delta(t) \in H(\curl;\Omega), \qquad
\curl \tilde{E}^\delta(t) = -\partial_t \tilde{B}^\delta(t) \mbox{ in } L^2(\Omega)^3.
\end{equation}
Now apply the Green identity for $H(\curl)$ to $\tilde E^\delta(t)\in H(\curl)$:
for any $\phi\in H(\curl;\Omega)$,
\[
\big(\tilde E^\delta(t),\ \curl\phi\big)_{L^2}
= \big(\curl\tilde E^\delta(t),\ \phi\big)_{L^2}
\;-\;\big\langle n\times \tilde E^\delta(t),\ \phi\big\rangle_{\partial\Omega}.
\]
Using \eqref{eq:aux2} and \eqref{eq:baba}, we find
\[
-\big(\partial_t \tilde B^\delta(t),\ \phi\big)_{L^2}
= \big(\curl\tilde E^\delta(t),\ \phi\big)_{L^2}
\;-\;\big\langle n\times \tilde E^\delta(t),\ \phi\big\rangle_{\partial\Omega}
= -\big(\partial_t \tilde B^\delta(t),\ \phi\big)_{L^2}
\;-\;\big\langle n\times \tilde E^\delta(t),\ \phi\big\rangle_{\partial\Omega}.
\]
Canceling the equal $L^2$ terms yields
\[
\big\langle n\times \tilde E^\delta(t),\ \phi\big\rangle_{\partial\Omega}=0
\qquad\forall\,\phi\in H(\curl;\Omega).
\]
Thus the tangential trace vanishes, $n\times \tilde E^\delta(t)=0$ in
the trace sense, i.e., $\tilde E^\delta(t)\in H_0(\curl;\Omega)$ on $t \in (\delta, T-\delta)$. 

Similarly, from \eqref{eq:reg-eqs}$_1$ 
and the fact that $\partial_t\tilde E^\delta(t),\tilde E^\delta(t)\in L^2(\Omega)^3$, we obtain that 
\[
\mu^{-1}\tilde B^\delta(t)\in H(\curl;\Omega)\qquad\text{for a.e.\ }t\in(\delta,T-\delta).
\]

\paragraph{Step 3 (energy identity with a time cutoff and \(\delta\downarrow0\)).} 
Fix a $\eta\in C_c^\infty(0,T)$ and set
\[
K:=\operatorname{supp}\eta
=\overline{\{\,s\in(0,T):\eta(s)\neq 0\,\}}\Subset(0,T).
\]
Since $\eta$ has compact support in $(0,T)$, $K$ is compact and $K\subset(0,T)$.
Let
\[
a:=\inf K,\qquad b:=\sup K,
\]
so $0<a\le b<T$ and $K\subset[a,b]$. Define the distances to the endpoints
\[
d_0:=a>0,\qquad d_T:=T-b>0,
\]
and set
\[
\delta_\eta:=\tfrac12\min\{d_0,d_T\}>0.
\]
Then for any $\delta\in(0,\delta_\eta)$ and any $s\in K$ we have
\[
s\ge a>d_0/2>\delta,
\]
and, since $b=T-d_T$ and $\delta<d_T/2<d_T$,
\[
s\le b=T-d_T<T-\delta.
\]
Hence $s\in(\delta,T-\delta)$, and therefore
\[
K\subset(\delta,T-\delta).
\]

Since \(\tilde E^\delta(\cdot)\in H_0(\curl;\Omega)\) and \(\mu^{-1}\tilde B^\delta(\cdot)\in H(\curl;\Omega)\) a.e.\ on \((\delta,T-\delta)\) (Step~2), we may use the time–dependent test functions
\[
\psi(t):=\eta(t)\,\tilde E^\delta(t)\in H_0(\curl;\Omega),\qquad
\phi(t):=\eta(t)\,\mu^{-1}\tilde B^\delta(t)\in H(\curl;\Omega)
\]
in \eqref{eq:reg-eqs}, integrate in \(t\in(0,T)\), and add the two relations. Because \(\eta=\eta(t)\) has no spatial dependence, the cross terms
\[
-(\mu^{-1}\tilde B^\delta,\curl(\eta\tilde E^\delta))
+(\curl \tilde E^\delta,\eta\,\mu^{-1}\tilde B^\delta)
= -\eta(\mu^{-1}\tilde B^\delta,\curl \tilde E^\delta)
+\eta(\curl \tilde E^\delta,\mu^{-1}\tilde B^\delta)=0 
\]
 cancel pointwise. 
Using that \(\tilde E^\delta,\tilde B^\delta\in C^\infty(\mathbb R;L^2(\Omega)^3)\) in time,
\[
\langle \varepsilon \partial_t \tilde E^\delta,\tilde E^\delta\rangle
=\tfrac12\,\tfrac{d}{dt}(\varepsilon\tilde E^\delta,\tilde E^\delta),\qquad
\langle \partial_t \tilde B^\delta,\mu^{-1}\tilde B^\delta\rangle
=\tfrac12\,\tfrac{d}{dt}(\mu^{-1}\tilde B^\delta,\tilde B^\delta),
\]
we obtain
\begin{equation}\label{eq:energy-cutoff-delta}
-\int_0^T \eta'(t)\,\mathcal{E}(\tilde E^\delta,\tilde B^\delta)(t)\,dt
+\int_0^T \eta(t)\,(\sigma \tilde E^\delta,\tilde E^\delta)\,dt=0,
\end{equation}
where \(\mathcal{E}(E,B)(t):=\tfrac12\big(\|\sqrt\varepsilon\,E(t)\|_{L^2}^2+\|\sqrt{\mu^{-1}}\,B(t)\|_{L^2}^2\big)\).

\emph{Passage \(\delta\downarrow0\) with \(\eta\) fixed.}
Because \(K=\mathrm{supp}\,\eta\Subset(0,T)\) and \(0<\delta<\delta_\eta\), the mollifications satisfy
\[
\tilde E^\delta\to \tilde E,\qquad \tilde B^\delta\to \tilde B
\quad\text{in }L^2\big(K;L^2(\Omega)^3\big)\ \text{as }\delta\downarrow0,
\]
Hence,
\[
\|\tilde E^\delta(\cdot)\|_{L^2}^2\ \to\ \|\tilde E(\cdot)\|_{L^2}^2,\qquad
\|\tilde B^\delta(\cdot)\|_{L^2}^2\ \to\ \|\tilde B(\cdot)\|_{L^2}^2
\quad\text{in }L^1(K),
\]
by the elementary bound \(|a^2-b^2|\le(|a|+|b|)|a-b|\). 
Therefore,
\[
\int_0^T \eta'(t)\,\mathcal E(\tilde E^\delta,\tilde B^\delta)(t)\,dt
\ \longrightarrow\
\int_0^T \eta'(t)\,\mathcal E(\tilde E,\tilde B)(t)\,dt,
\]
and, since \(\sigma\in L^\infty\) and \(\tilde E^\delta\to \tilde E\) in \(L^2(K;L^2)\),
\[
\int_0^T \eta(t)\,(\sigma \tilde E^\delta,\tilde E^\delta)\,dt
\ \longrightarrow\
\int_0^T \eta(t)\,(\sigma \tilde E,\tilde E)\,dt.
\]
Letting \(\delta\downarrow0\) in \eqref{eq:energy-cutoff-delta} yields, for our fixed \(\eta\in C_c^\infty(0,T)\),
\begin{equation}\label{eq:energy-cutoff-limit}
-\int_0^T \eta'(t)\,\mathcal{E}(\tilde E,\tilde B)(t)\,dt
+\int_0^T \eta(t)\,(\sigma \tilde E,\tilde E)\,dt=0.
\end{equation}
Since the right–hand side belongs to \(L^1(0,T)\), \eqref{eq:energy-cutoff-limit} shows that
\(\mathcal{E}(\tilde E,\tilde B)\in W^{1,1}(0,T)\) with
\[
\frac{d}{dt}\,\mathcal{E}(\tilde E,\tilde B)(t)=-(\sigma \tilde E(t),\tilde E(t))
\quad\text{in }\mathcal D'(0,T),
\]
hence \(\mathcal{E}(\tilde E,\tilde B)\) is absolutely continuous on \([0,T]\). Integrate both sides to conclude that $\tilde{E} = 0 $ and $ \tilde{B} =0$. The proof is complete.
\end{proof}

The same argument from the above result can be used to show that the solution to \eqref{eq:MaxWeak} is in fact continuous in time. 
We state an auxiliary result before proving this, see \cite[Proposition 2.5.1]{JDroniou_2001a} for details. 

\begin{lemma}[Differentiability of Scalar Pairings]
\label{lem:pairing_differentiability}
Let \( V \) be a reflexive Banach space with dual \( V^* \), and let \(u \in H^1(0,T; V^*)\). 
Then, for every \( v \in V \), the scalar function
\[
\alpha(t) := \langle u(t), v \rangle_{V^*, V} \in H^1(0,T) , 
\]
and its derivative satisfies
\[
\frac{d}{dt} \alpha(t) = \langle \partial_t u(t), v \rangle_{V^*, V} \quad \text{for a.e. } t \in (0,T).
\]
\end{lemma}

The next result shows time continuity of solution to \eqref{eq:MaxWeak}.
\begin{corollary}[Strong $L^2$–continuity in time and energy identity]\label{cor:C0-in-time}
Let $(E,B)$ be the weak solution of \eqref{eq:MaxWeak} in the sense of Definition~\ref{def:weak_soln}.  
Then $E,B\in C([0,T];L^2(\Omega)^3)$ and, for every $t\in[0,T]$,
\begin{equation}\label{eq:energy-identity-inhom}
\mathcal E(t)+\int_0^t (\sigma E(s),E(s))\,ds
=\mathcal E(0)+\int_0^t (f(s),E(s))\,ds,
\qquad
\mathcal E(t):=\tfrac12\big(\|\sqrt\varepsilon\,E(t)\|_{L^2}^2+\|\sqrt{\mu^{-1}}\,B(t)\|_{L^2}^2\big).
\end{equation}
In particular, $t\mapsto \|E(t)\|_{L^2}$ and $t\mapsto \|B(t)\|_{L^2}$ are continuous on $[0,T]$.
\end{corollary}

\begin{proof}
\emph{Step 1 (weak $L^2$–continuity).}
The continuous embedding $H_0(\curl;\Omega)\hookrightarrow L^2(\Omega)^3$ induces
$L^2(\Omega)^3\hookrightarrow H_0(\curl;\Omega)^*$ via
$\langle E(t),w\rangle_{H_0(\curl)}=(E(t),w)_{L^2}$.
Since $E\in H^1(0,T;H_0(\curl)^*)$, Lemma~\ref{lem:pairing_differentiability} yields
$(E(\cdot),w)_{L^2}\in H^1(0,T)\subset C([0,T])$ for each $w\in H_0(\curl;\Omega)$.
Because $H_0(\curl;\Omega)\supset C_c^\infty(\Omega)^3$ is dense in $L^2(\Omega)^3$ and
$\sup_{t}\|E(t)\|_{L^2}<\infty$, we approximate any $\phi\in L^2$ by $w_k\in H_0(\curl)$ and pass
to the limit uniformly in $t$ to conclude that $t\mapsto (E(t),\phi)_{L^2}$ is continuous.
Thus $E\in C_w([0,T];L^2)$, i..e, weakly continuous. The same argument with $H(\curl;\Omega)$ shows $B\in C_w([0,T];L^2)$.

\emph{Step 2 (energy identity in distribution form).}
Repeating the interior time–mollification/cutoff test used in the uniqueness proof, now retaining
the forcing term, gives for every $\eta\in C_c^\infty(0,T)$
\[
-\int_0^T \eta'(t)\,\mathcal E(t)\,dt
+\int_0^T \eta(t)\,(\sigma E,E)\,dt
=\int_0^T \eta(t)\,(f,E)\,dt .
\]
Since $f\in L^2(0,T;L^2)$ and $E\in L^\infty(0,T;L^2)$, the right-hand side is in $L^1(0,T)$, hence
$\mathcal E\in W^{1,1}(0,T)$ with
$\mathcal E'(t)=(f(t),E(t))-(\sigma E(t),E(t))$ a.e. Integrating from $0$ to $t$ yields
\eqref{eq:energy-identity-inhom}.

\emph{Step 3 (strong $L^2$–continuity).}
From Step~1, $E,B\in C_w([0,T];L^2)$. From Step~2, $t\mapsto\|\sqrt{\varepsilon}E(t)\|_{L^2}$ and
$\|\sqrt{\mu^{-1}}B(t)\|_{L^2}$ are continuous; by uniform ellipticity and boundedness of
$\varepsilon,\mu^{-1}$ these norms are equivalent to $\|E(t)\|_{L^2}$ and $\|B(t)\|_{L^2}$.
In a Hilbert space, weak continuity plus continuity of the norm implies strong continuity.
Hence $E,B\in C([0,T];L^2(\Omega)^3)$.
\end{proof}

Now we are ready to state our existence of solution proof.

\begin{theorem}[Well-posedness: Existence, Uniqueness, and Continuous Dependence]
\label{thm:main}
Let \( \Omega \subset \mathbb{R}^3 \) be a bounded Lipschitz domain and \( T > 0 \). Under the Assumption~\ref{ass:data}, there exists a unique weak solution \( (E, B) \), according to 
Definition~\ref{def:weak_soln}, to the Maxwell system \eqref{eq:Maxwell_strong}. Furthermore, the solution 
satisfies the stability estimate
\begin{equation}\label{eq:apriori_bound}
\begin{aligned}
\| \partial_t E \|_{L^2(0,T;H_0(\mathrm{curl};\Omega)^*)}  
&+ \| \partial_t B \|_{L^2(0,T;H(\curl; \Omega)^*)} 
+  \|E\|_{C([0,T];L^2(\Omega)^3)} + \|B\|_{C([0,T];L^2(\Omega)^3)} \\
&\le C \left( \|f\|_{L^2(0,T;L^2(\Omega)^3)} + \|E_0\|_{L^2(\Omega)^3} + \|B_0\|_{L^2(\Omega)^3} \right),
\end{aligned}
\end{equation}
for some constant \( C \) depending only on \( T, \varepsilon, \mu \).
\end{theorem}

\begin{proof}
Uniqueness of the weak solutions from Definition~\ref{def:weak_soln} is due to Theorem~\ref{thm:uniq-cont} and $C([0,T];L^2(\Omega)^3)$ regularity is due to Corollary~\ref{cor:C0-in-time}. Next we will establish existence and continuous dependence. 

\textbf{Step 1: Galerkin Approximation.}  
Since \( H_0(\mathrm{curl}; \Omega) \) 
and $H(\curl;\Omega)$ are separable Hilbert spaces, there exists basis \( \{\phi_i\}_{i=1}^\infty \subset H(\curl;\Omega) \)  and $\{\psi_i\}_{i=1}^\infty \subset H_0(\mathrm{curl}; \Omega)$, which can be made orthonormal in $L^2(\Omega)^3$ (e.g., using Gram–Schmidt). Notice that the resulting vectors (after Gram–Schmidt in $L^2$) still forms a basis of $H_0(\mathrm{curl}; \Omega)$ and $H(\curl;\Omega)$. 

We use the relation $B = \mu H$. For each \( N \in \mathbb{N} \), we define the Galerkin approximations:
\[
E_N(x,t) := \sum_{i=1}^N \alpha_i^N(t)\, \psi_i(x), \qquad
H_N(x,t) := \sum_{i=1}^N \beta_i^N(t)\, \phi_i(x) . 
\]
and $B_N = \mu H_N$. We require \( (E_N, H_N) \) to satisfy the Galerkin system:
\begin{equation} \label{eq:Galerkin}
\begin{aligned}
\langle \varepsilon\, \partial_t E_N, \psi_j \rangle_{H_0(\mathrm{curl};\Omega)} + (\sigma E_N, \psi_j) - ( H_N, \mathrm{curl}\, \psi_j) &= (f, \psi_j), \\
\langle \mu \partial_t H_N, \phi_j \rangle_{H(\curl;\Omega)} + ( \mathrm{curl}\, E_N,  \phi_j ) &= 0,
\end{aligned}
\end{equation}
for all \( j = 1, \dots, N \), with initial data projections:
\[
\alpha_j^N(0) := (E_0, \psi_j), \qquad 
\beta_j^N(0)  := (H_0, \phi_j) ,
\]
where $H_0 = \mu^{-1} B_0 \in L^2(\Omega)^3$ because $\mu^{-1} \in L^\infty(\Omega;\mathbb{R}^{3\times 3})$. 
This leads to a system of ODEs for the coefficients with $j = 1,\dots, N$:
\[
\begin{aligned}
\sum_{i=1}^N \left[
(\varepsilon \psi_i, \psi_j) \dot{\alpha}_i^N(t)
+ (\sigma \psi_i, \psi_j) \alpha_i^N(t)
- (\phi_i, \mathrm{curl} \, \psi_j) \beta_i^N(t)
\right] &= (f(\cdot,t), \psi_j), \\
\sum_{i=1}^N \left[
(\mu \phi_i,  \phi_j) \dot{\beta}_i^N(t)
+ (\mathrm{curl}\, \psi_i, \phi_j) \alpha_i^N(t)
\right] &= 0.
\end{aligned}
\]
Define the matrices:
\[
\begin{aligned}
[M_E]_{ij} &:= (\varepsilon \psi_i, \psi_j), \qquad
[K_E]_{ij} := (\sigma \psi_i, \psi_j), \\
[C]_{ij} &:= (\phi_i, \mathrm{curl}\, \psi_j), \qquad
[M_B]_{ij} := (\mu\phi_i,  \phi_j),
\end{aligned}
\]
and vectors:
\[
\begin{aligned}
\alpha^N(t) &:= [\alpha_1^N(t), \dots, \alpha_N^N(t)]^\top, \qquad
\beta^N(t) := [\beta_1^N(t), \dots, \beta_N^N(t)]^\top, \\
F(t) &:= [(f(\cdot,t), \psi_1), \dots, (f(\cdot,t), \psi_N)]^\top.
\end{aligned}
\]
The Galerkin system can now be written compactly as:
\begin{equation} \label{eq:Galerkin2}
\begin{aligned}
M_E \dot{\alpha}^N(t) + K_E \alpha^N(t) - C \beta^N(t) &= F(t), \\
M_B \dot{\beta}^N(t) + C^\top \alpha^N(t) &= 0.
\end{aligned}
\end{equation}
Define the combined unknown vector:
\[
y(t) := \begin{bmatrix} \alpha^N(t) \\ \beta^N(t) \end{bmatrix}, \qquad
y(0) := \begin{bmatrix} \alpha^N(0) \\ \beta^N(0) \end{bmatrix}.
\]
Define the block matrix and forcing:
\[
A := 
\begin{bmatrix}
- M_E^{-1} K_E & M_E^{-1} C \\
- M_B^{-1} C^\top & 0
\end{bmatrix}, \qquad
G(t) := 
\begin{bmatrix}
M_E^{-1} F(t) \\
0
\end{bmatrix}.
\]
Then the Galerkin ODE system reads:
\[
\frac{d}{dt} y(t) = A y(t) + G(t), \qquad y(0) = y_0.
\]
By the Carath\'eodory existence theorem of ODEs \cite[Theorem~5.2]{JKHale_1980a} 
for systems with \( A \in \mathbb{R}^{2N \times 2N} \) constant and \( G \in L^2(0,T; \mathbb{R}^{2N}) \), we obtain: 
\[
y \in H^1(0,T; \mathbb{R}^{2N}) \quad \Rightarrow \quad \alpha_i^N, \beta_i^N \in H^1(0,T).
\]
Thus,
\[
(E_N, H_N) \in H^1(0,T; Y_N) \times H^1(0,T; X_N),
\]
where \( Y_N := \mathrm{span}\{\psi_1, \dots, \psi_N\} \subset H_0(\mathrm{curl};\Omega) \), and \( X_N := \mathrm{span}\{\phi_1, \dots, \phi_N\} \subset H(\mathrm{curl};\Omega) \).

\medskip
\textbf{Step 2: Energy Estimate.}  
We now derive a uniform a priori energy estimate for the Galerkin approximations \( (E_N, B_N) \). Recall the Galerkin system from \eqref{eq:Galerkin}. 
Multiply the first equation by \( \alpha_j^N(t)  \) and sum over \( j = 1, \dots, N \). Using the expansion \( E_N = \sum_{j=1}^N \alpha_j^N(t) \psi_j \), this yields:
\[
\langle \varepsilon \partial_t E_N, E_N \rangle_{H_0(\mathrm{curl};\Omega)} + (\sigma E_N, E_N) - ( H_N, \mathrm{curl}\,  E_N) = (f, E_N).
\]
Similarly, multiply the second equation by \( \beta_j^N(t) \), sum over \( j = 1, \dots, N \), and use the expansion
\(
H_N = \sum_{j=1}^N \beta_j^N(t) \phi_j
\)
to obtain:
\[
\langle \mu \partial_t H_N, H_N \rangle_{H(\mathrm{curl};\Omega)} + ( \mathrm{curl}\, E_N, H_N ) = 0.
\]
Adding the two equations gives the energy identity:
\[
(\varepsilon \partial_t E_N, E_N) + (\mu \partial_t H_N, H_N) + (\sigma E_N, E_N) = (f, E_N) , 
\]
where we have used the fact that $\mu \partial_t H_N \in L^2(\Omega)^3$ and $\varepsilon \partial_t B_N \in L^2(\Omega)^3$ therefore the duality $\langle \cdot, \cdot \rangle$ coincides with $L^2$ pairing $(\cdot,\cdot)$. 
Using the identity \( (u', u) = \frac{1}{2} \frac{d}{dt} \|u\|^2_{L^2(\Omega)} \), we get:
\[
\frac{1}{2} \frac{d}{dt} \left( \|\sqrt{\varepsilon} E_N\|_{L^2}^2 + \|\sqrt{\mu} H_N\|_{L^2}^2 \right) + \|\sqrt{\sigma} E_N\|_{L^2}^2 = (f, E_N).
\]
Apply the Cauchy--Schwarz and Young inequalities:
\[
(\varepsilon^{-1} f, \varepsilon E_N) \le \|\sqrt{\varepsilon^{-1}}f\|_{L^2} \|\sqrt{\varepsilon}E_N\|_{L^2} 
\le \frac{1}{2\varepsilon_0} \|f\|_{L^2}^2 + \frac{1}{2} \|\sqrt{\varepsilon} E_N\|_{L^2}^2 . 
\]
We obtain that 
\[
\frac{1}{2}\frac{d}{dt} \left( \|\sqrt{\varepsilon} E_N\|_{L^2}^2 + \|\sqrt{\mu} H_N\|_{L^2}^2 \right) \le C \|f\|_{L^2}^2 
+ 
\left( \|\sqrt{\varepsilon} E_N\|_{L^2}^2 + \|\sqrt{\mu} H_N\|_{L^2}^2 \right).
\]
Define the energy functional:
\[
\mathcal{E}(E,H)(t) := \frac{1}{2} \left( \|\sqrt{\varepsilon} E(t)\|_{L^2}^2 + \|\sqrt{\mu} H(t)\|_{L^2}^2 \right) , 
\]
for a.e. $t\in [0,T)$. 
We obtain the differential inequality:
\[
\frac{d}{dt} \mathcal{E}(E_N,H_N)(t) \le \mathcal{E}(E_N,H_N)(t) + C \|f(t)\|_{L^2}^2,
\]
Apply Gr\"onwall's estimate in differential form, we obtain that 
\[
\mathcal{E}(E_N,H_N)(t) \le C \left( \mathcal{E}(E_N,H_N)(0) + \int_0^t \|f(s)\|_{L^2}^2 \, ds \right).
\]
Using the initial data projections:
\[
\mathcal{E}(E_N,H_N)(0) = \frac12 \left( \|\sqrt{\varepsilon} E_N(0)\|_{L^2}^2 + \|\sqrt{\mu} H_N(0)\|_{L^2}^2 \right) \le C \left( \|E_0\|_{L^2}^2 + \|B_0\|_{L^2}^2 \right),
\]
we obtain the uniform energy bound:
\begin{equation} \label{eq:energy_estimate}
\begin{aligned}
\|E_N\|_{L^\infty(0,T; L^2)}^2 + \|H_N\|_{L^\infty(0,T; L^2)}^2 
&\le C \left( \|E_0\|_{L^2}^2 + \|B_0\|_{L^2}^2 + \|f\|^2_{L^2(0,T;L^2)} \right),
\end{aligned}
\end{equation}
where \( C > 0 \) depends only on $L^\infty$ bounds for \( \varepsilon, \mu^{-1} \), and the final time \( T \), but not on \( N \). 
In particular, the energy estimate implies (up to subsequences)
\begin{equation}\label{eq:limits1}
\begin{aligned}
E_N &\rightharpoonup E \quad \text{weakly-* in } L^\infty(0,T;L^2(\Omega)^3), \\
E_N &\rightharpoonup E \quad \text{weakly in } L^2(0,T;L^2(\Omega)^3), \\
H_N &\rightharpoonup H \quad \text{weakly-* in } L^\infty(0,T;L^2(\Omega)^3) .   
\end{aligned}
\end{equation}
Recall that since the norm is convex and continuous, it is therefore is weakly lower-semicontinuous. Then using \eqref{eq:limits1} in \eqref{eq:energy_estimate} we obtain the bound \eqref{eq:apriori_bound}, except the time derivative part.

\medskip
\textbf{Step 3: Convergence of \(\partial_t E_N, \,    \partial_t H_N  \).} 
Since \( H_0(\mathrm{curl};\Omega) \) and \( H(\mathrm{curl};\Omega) \) are Hilbert spaces, they admit orthogonal projections (e.g., Riesz projection given in Proposition~\ref{prop:Riesz-Hcurl}):
\begin{equation}\label{eq:projection0}
\begin{aligned}
    \Pi_N : H(\mathrm{curl};\Omega) \rightarrow X_N := \mathrm{span}\{ \phi_i \}_{i=1}^N, 
    &\qquad 
    \Pi_N' : H_0(\mathrm{curl};\Omega) \rightarrow Y_N := \mathrm{span}\{ \psi_i \}_{i=1}^N, \\
    \Pi_N \phi \rightarrow \phi \text{ in } H(\mathrm{curl};\Omega) \quad \forall \phi \in H(\mathrm{curl};\Omega), 
    &\qquad 
    \Pi_N' \psi \rightarrow \psi \text{ in } H_0(\mathrm{curl}) \quad \forall \psi \in H_0(\mathrm{curl}; \Omega).
\end{aligned}    
\end{equation}
Recall that for a Hilbert space $X$, we have that \(C_c^\infty(0,T) \otimes X \) is dense in \(L^2(0,T;X)\) (cf.~\cite[Corollaire 1.3.1]{JDroniou_2001a}). 
Therefore for an arbitrary \(\phi \in L^2(0,T; H(\mathrm{curl};\Omega))\), we can write \(\phi = w(x) \eta(t)\) with \(w \in H(\mathrm{curl};\Omega)\) and \(\eta \in C_c^\infty(0,T)\) which are arbitrary.

Consider \eqref{eq:Galerkin}$_{2}$, we have that 
\[
\begin{aligned}
\int_0^T \langle \mu \partial_t H_N , \Pi_N w \rangle_{ H(\mathrm{curl};\Omega)} \eta(t)\, dt 
&= - \int_0^T (\mathrm{curl}\, E_N , \Pi_N w) \eta(t) \, dt 
\quad \forall w \in  H(\mathrm{curl};\Omega), \, \forall \eta \in C_c^\infty(0,T).
\end{aligned}
\]
Applying integration-by-parts in-time on the left-hand side and in-space (with vanishing tangential trace for $E_N$) on the right-hand side, we obtain 
\[
\begin{aligned}
-\int_0^T ( \mu H_N , \Pi_N w ) \eta'(t)\, dt &= - \int_0^T (E_N , \mathrm{curl}\, \Pi_N w) \eta(t) \, dt 
\quad \forall w \in  H(\mathrm{curl};\Omega), \, \forall \eta \in C_c^\infty(0,T).
\end{aligned}
\]
Now using \eqref{eq:limits1} and \eqref{eq:projection0} we deduce 
\[
\begin{aligned}
-\int_0^T ( \mu H , w  ) \eta'(t)\, dt 
&= - \int_0^T (E , \curl w) \eta(t) \, dt \\
&=: \int_0^T \langle \mathcal{L}(t), w \rangle_{ H(\mathrm{curl};\Omega)} \eta(t) \, dt
\quad \forall w \in  H(\mathrm{curl};\Omega) \, , \forall \eta \in C_c^\infty(0,T) . 
\end{aligned}
\]
We have that for a.e., $t \in (0,T)$
\[
|\langle \mathcal{L}(t), w \rangle_{ H(\mathrm{curl};\Omega)}| 
= |(E , \curl w)|
\le C \| E(t) \|_{L^2(\Omega)^3} \| w \|_{ H(\mathrm{curl};\Omega)} . 
\]
Since $E \in L^2(0,T;L^2(\Omega)^3)$, therefore \(\mathcal{L}(t) :  H(\mathrm{curl};\Omega) \rightarrow  H(\mathrm{curl};\Omega)^*\) is bounded and linear and \( \| \mathcal{L}(t) \|_{ H(\mathrm{curl};\Omega)^*} \in L^2(0,T) \). 
Thus we have that 
\[
\int_0^T \langle \mathcal{L}(t), w \rangle_{ H(\mathrm{curl};\Omega)} \eta(t) \, dt 
= -\int_0^T ( \mu H , w  ) \eta'(t)\, dt  
\quad \forall w \in  H(\mathrm{curl};\Omega) \, , \forall \eta \in C_c^\infty(0,T) . 
\]
Since $ H(\mathrm{curl};\Omega)$ is dense in $L^2(\Omega)^3$,   
using \cite[Section 7.2]{TRoubicek_2005a},  
we can identify $(\mu H,w)$ as $\langle \mu H, w \rangle_{ H(\mathrm{curl};\Omega)}$ a.e. in $t \in (0,T)$. Then from the definition of weak derivative, we deduce that 
\[
 \mathcal{L}(t) = \mu \partial_t H \in L^2(0,T; H(\mathrm{curl};\Omega)^*)
\]

This proves that for a.e. $t$
\[
    \langle \mu \partial_t H , \phi \rangle_{ H(\mathrm{curl};\Omega)}
    +(E , \curl \phi) = 0 , \quad 
    \forall \phi \in  H(\mathrm{curl};\Omega) . 
\]
A similar argument using \(C_c^\infty(0,T) \otimes H_0(\mathrm{curl}; \Omega)\) shows that
\[
\partial_t E \in L^2(0,T; H_0(\mathrm{curl}; \Omega)^*),
\]
and the first equation of \eqref{eq:MaxWeak} holds.

Moreover, the following a priori bounds hold:
\[
\begin{aligned}
\| \mu\partial_t H \|_{L^2(0,T; H(\curl; \Omega)^*)} &\le C \left( \|f\|_{L^2(0,T;L^2(\Omega)^3)} + \|E_0\|_{L^2(\Omega)^3} + \|B_0\|_{L^2(\Omega)^3} \right), \\
\| \partial_t E \|_{L^2(0,T; H_0(\mathrm{curl}; \Omega)^*)} &\le C \left( \|f\|_{L^2(0,T;L^2(\Omega)^3)} + \|E_0\|_{L^2(\Omega)^3} + \|B_0\|_{L^2(\Omega)^3} \right).
\end{aligned}
\]

\medskip
\textbf{Step 4: Convergence of Initial Conditions.}  
Since \( \{\phi_i\}_{i=1}^\infty \) is an orthonormal basis of \( L^2(\Omega)^3 \), the projection
\[
H_N(0) := \sum_{i=1}^N (H_0, \phi_i)\, \phi_i \to H_0 \quad \text{in } L^2(\Omega)^3 \quad \text{as } N \to \infty.
\]
Since $B_N = \mu H_N$, we immediately get that $B_N(0) = \mu H_N(0) \rightarrow B_0 = \mu H_0$ in $L^2(\Omega)^3$. Similarly, we can argue for $E_N$. All the above regularity results directly transfer from $H$ to $B$. 
\end{proof}

In view of Theorem~\ref{thm:main}, and according to  Definition~\ref{def:weak_soln}, it then remains to show that $\mathrm{div }\, B = 0$ in $L^2(\Omega)$ and a.e. $t\in [0,T)$.

\begin{proposition}[Divergence-Free Evolution]
\label{prop:divB}
Let $(E,B)$ be weak solution according to Definition~\ref{def:weak_soln}. 
Assume further that the initial magnetic field satisfies \( \mathrm{div}\, B_0 = 0 \) in \( L^2(\Omega) \). Then it follows that
\[
\mathrm{div}\, B(\cdot,t) = 0 \qquad \text{in } \mathcal{D}'(\Omega) \text{ for all } t \in [0,T),
\]
and thus 
\[
B \in L^\infty(0,T; H(\mathrm{div}^0;\Omega)). 
\]
\end{proposition}

\begin{proof}
Let \( \psi \in C_c^\infty(\Omega) \). We define the divergence pairing via duality as
\[
\langle \mathrm{div}\, B(\cdot,t), \psi \rangle := -\langle B(\cdot,t), \nabla \psi \rangle.
\]
Since \( B \in H^1(0,T; H(\curl;\Omega)^*) \), the mapping \( t \mapsto \langle B(\cdot,t), \nabla \psi \rangle \in H^1(0,T) \) according to Lemma~\ref{lem:pairing_differentiability}. Thus, we may compute its time derivative:
\[
\frac{d}{dt} \langle \mathrm{div}\, B(\cdot,t), \psi \rangle = -\frac{d}{dt} \langle B(\cdot,t), \nabla \psi \rangle = -\langle \partial_t B(\cdot,t), \nabla \psi \rangle.
\]
But \( \nabla \psi \in H(\mathrm{curl};\Omega) \), and we have the distributional identity \( \mathrm{curl}(\nabla \psi) = 0 \). Therefore, by the second equation in \eqref{eq:MaxWeak}, we obtain
\[
\langle \partial_t B(\cdot,t), \nabla \psi \rangle = - (E(\cdot,t), \mathrm{curl}\, \nabla \psi) = 0.
\]
Hence,
\[
\frac{d}{dt} \langle \mathrm{div}\, B(\cdot,t), \psi \rangle = 0 \quad \forall \psi \in C_c^\infty(\Omega).
\]
This implies that the map \( t \mapsto \langle \mathrm{div}\, B(\cdot,t), \psi \rangle \) is constant. Since \( \mathrm{div}\, B_0 = 0 \) in \( L^2(\Omega) \), so \( \mathrm{div}\, B_0 = 0 \) in \(\mathcal{D}'(\Omega) \), we have
\[
\langle \mathrm{div}\, B(\cdot,t), \psi \rangle = \langle \mathrm{div}\, B_0, \psi \rangle = 0 \quad \forall \psi \in C_c^\infty(\Omega),\ \mbox{a.e. } t \in [0,T).
\]
Thus \( \mathrm{div}\, B(\cdot,t) = 0 \) in the distributional sense, a.e. \( t \in [0,T) \). Next, we notice that 
\[
\langle \mathrm{div}\, B(\cdot,t), \psi \rangle
= -(B(\cdot,t) , \nabla \psi) = 0 = (0,\psi) , \quad 
\forall \psi \in C_c^\infty(\Omega), \mbox{ a.e. } t \in [0,T)
\]
which implies that $\divv B(t) = 0$ in $L^\infty(0,T;L^2(\Omega))$ and the proof is complete. 
\end{proof}

\section{N\'ed\'elec/RT spaces and (discrete) de Rham sequence}
\label{sec:semidiscrete}

Having shown the well-posedness of the continuous problem, we now turn 
our attention to the semi-discrete (in-space) approximation of the weak form \eqref{eq:MaxWeak}. We will also establish that the solution to the semi-discrete problem converges to the solution to the continuous problem. Notice that due to the continuous time nature, these results are agnostic to any particular time discretization. 
From hereon we will assume that the $\Omega$ is Lipschitz, simply connected,  with connected boundary. Though this is only used in Lemma~\ref{lem:simple-B0} and all other results are true for Lipschitz domains.

Let $\{\mathcal T_h\}_{h}$ be a shape-regular tetrahedral family of meshes of $\Omega$.
Fix an order $k\ge 0$.
\begin{itemize}
\item The \emph{N\'ed\'elec} space conforming to $H(\curl)$ space is given by:
\[
\mathcal N_h := \{\,v_h\in H(\curl;\Omega): v_h|_K\in \mathcal N_k(K)\ \forall K\in\mathcal T_h\,\},
\]
and $\mathcal N_h^0:=\mathcal N_h\cap H_0(\curl;\Omega)$.
\item The \emph{Raviart--Thomas} space (conforming to $H(\divv)$):
\[
\mathcal{RT}_h := \{\,w_h\in H(\divv;\Omega): w_h|_K\in \mathcal{RT}_k(K)\ \forall K\,\}.
\]
\item Piecewise polynomial scalars $\mathcal Q_h:=\{q_h\in L^2(\Omega): q_h|_K\in \mathbb P_k(K)\}$.
\end{itemize}

\paragraph{Discrete de Rham sequence (exactness).}
There exist commuting diagrams and the exact sequence, on simply connected domains with connected boundary:
\[
H_0^1(\Omega)\xrightarrow{\ \nabla\ } H_0(\curl;\Omega)\xrightarrow{\ \curl\ } H(\divv;\Omega)\xrightarrow{\ \divv\ } L^2(\Omega)\to 0,
\]
and discretely
\[
\mathcal P_h \xrightarrow{\ \nabla\ } \mathcal N_h^0 \xrightarrow{\ \curl\ } \mathcal{RT}_h \xrightarrow{\ \divv\ } \mathcal Q_h \to 0,
\]
where $\mathcal P_h$ are conforming Lagrange spaces. In particular,
\[
\curl\mathcal N_h^0 \subset \mathcal{RT}_h,\qquad \divv(\curl \mathcal N_h^0)=0.
\]

\paragraph{Global $L^2$ projector onto $\mathcal{RT}_h$.}
Let $P_h:L^2(\Omega)^3\to \mathcal{RT}_h$ be the $L^2$-orthogonal projector:
\begin{equation}\label{eq:projection}
(u-P_h u, v_h)=0\quad\forall v_h\in \mathcal{RT}_h.
\end{equation}
$P_h$ is linear, idempotent, $\|P_h\|_{L^2\to L^2}=1$, and $P_h\Phi\to \Phi$ in $L^2$ for all $\Phi\in L^2(\Omega)^3$. See Proposition~\ref{prop:L2proj-RT} for a proof. 

\paragraph{Riesz projection on $H_0(\curl;\Omega)$.}
Define, for each $u\in H_0(\curl;\Omega)$, the Riesz projection $R_h u\in \mathcal N_h^0$ by
\begin{equation}\label{eq:def-Rh1}
  a(R_h u, v_h) \;=\; a(u, v_h)\qquad\forall\,v_h\in \mathcal N_h^0,
\end{equation}
where $a(u,v):=(u,v)+(\curl u,\curl v)$. Such $R_h u \in \mathcal{N}_h^0$ is unique for each $u$ and it fulfills the Galerkin orthogonality $a(u-R_h u, v_h) = 0$ for all $v_h \in \mathcal{N}_h^0$, stability $\|R_h u\|_{H(\curl;\Omega)} \le \|u\|_{H(\curl;\Omega)}$, and $\|u-R_hu\|_{H_0(\curl;\Omega)} \rightarrow 0$ as $h \rightarrow 0$. See Proposition~\ref{prop:Riesz-Hcurl} for a proof.

\subsection{Semi-discrete FE scheme}

\paragraph{Discrete unknowns/test spaces.}
We approximate
\[
E_h(t)\in \mathcal N_h^0,\qquad B_h(t)\in \mathcal{RT}_h,
\]
and for all times test Amp\`ere with $\psi_h\in \mathcal N_h^0$ and Faraday with $\phi_h\in \mathcal{RT}_h$.

\paragraph{Semi-discrete in space scheme.}
Find $(E_h(t),B_h(t)) \in \mathcal{N}_h^0 \times \RT_h$ such that for a.e.\ $t\in(0,T)$
\begin{subequations}\label{eq:SD}
\begin{align}
(\varepsilon \partial_t E_h(t),\psi_h) + (\sigma E_h(t),\psi_h) - (\mu^{-1}B_h(t),\curl\psi_h) &= (f(t),\psi_h)
&&\forall \psi_h\in \mathcal N_h^0, \label{eq:SD-amp}
\\
(\partial_t B_h(t),\phi_h) + (\curl E_h(t),\phi_h) &= 0
&&\forall \phi_h\in \mathcal{RT}_h. \label{eq:SD-far}
\end{align}
\end{subequations}
Initial data are chosen as 
\[
E_h(0)=Q_h^N E_0 \in \mathcal N_h^0,\qquad
\text{$B_h(0)\in \mathcal{RT}_h$ \quad with \quad }(\divv B_h(0),q_h)=0\ \forall q_h\in \mathcal Q_h,
\]
Here, $ Q_h^N: L^2(\Omega)^3\to \mathcal N_h^0$ is the $L^2$-orthogonal projector with $\|E_h(0)-E_0\|_{L^2(\Omega)^3} \;\xrightarrow[h\to0]{}\; 0$ and $\|B_h(0)-B_0\|_{L^2(\Omega)^3} \;\xrightarrow[h\to0]{}\; 0$. An example of such approximation of $B_h(0)$ is the $L^2$-orthogonal projection of $B_0$ onto the discrete divergence-free subspace 
\(Z_h := \{v_h\in\mathcal{RT}_h:\ \divv v_h=0\ \text{in }\mathcal Q_h\}\). See Proposition~\ref{prop:constrained-proj}. 

\begin{remark}[Well-posedness of \eqref{eq:SD}]
Fixing $h$, \eqref{eq:SD} is a linear ODE in finite dimensions with a positive definite (block) mass matrix; thus it has a unique solution with $E_h\in H^1(0,T;\mathcal N_h^0)$,\qquad $B_h\in H^1(0,T;\mathcal{RT}_h)$. 
\end{remark}

\begin{lemma}[Discrete Gauss law preservation]
\label{lem:disc-gauss-variational}
Let $(E_h,B_h)$ solve the semi-discrete \eqref{eq:SD-far}. 
Assume the compatible N\'ed\'elec/RT pair so that $\curl\Nc_h^0\subset \RT_h$ and $\divv:\RT_h\to \mathcal{Q}_h$ is the usual (onto) divergence operator. Then, for every $q_h\in\mathcal{Q}_h$,
\[
(\divv B_h(t),q_h)=(\divv B_h(0),q_h)\qquad\text{for a.e.\ }t\in[0,T).
\]
In particular, if $(\divv B_h(0),q_h)=0$ for all $q_h\in \mathcal{Q}_h$, then $(\divv B_h(t),q_h)=0$ for all $q_h\in \mathcal{Q}_h$ and a.e.\ $t$.
\end{lemma}

\begin{proof}
\emph{Step 1: Define the residual and show it vanishes.}
For a.e.\ $t$, define
\[
r_h(t):= \partial_t B_h(t)+\curl E_h(t).
\]
By construction, $B_h(t)\in\mathcal{RT}_h$ for all $t$ and the coefficient vector of $B_h$ is absolutely continuous in time. 
Indeed, by well-posedness of the semi-discrete system we have $B_h\in H^1(0,T;\mathcal{RT}_h)$, for example, write $B_h(t)=\sum_{j=1}^{N_B} b_j(t)\,\rho_j$ with $b_j\in H^1(0,T)$ and
$\{\rho_j\}$ a basis of $\mathcal{RT}_h$. 
Hence $\partial_t B_h(t)\in\mathcal{RT}_h$ for a.e.\ $t$. By discrete exactness, $\curl E_h(t)\in\mathcal{RT}_h$. Therefore $r_h(t)\in\mathcal{RT}_h$ for a.e.\ $t$.

From \eqref{eq:SD-far} we have
\[
(r_h(t),\phi_h)=0\qquad \forall\,\phi_h\in\mathcal{RT}_h.
\]
Thus $r_h(t)\in \mathcal{RT}_h^\perp$ (orthogonal complement in $L^2(\Omega)^3$). Since also $r_h(t)\in\mathcal{RT}_h$, we have
\[
r_h(t)\in \mathcal{RT}_h\cap \mathcal{RT}_h^\perp=\{0\}.
\]
Hence
\[
\partial_t B_h(t)+\curl E_h(t)=0\quad\text{in }L^2(\Omega)^3,\ \text{for a.e.\ }t.
\]

\emph{Step 2: Apply $\divv:\mathcal{RT}_h\to\mathcal{Q}_h$.}
Because $\divv$ maps $\mathcal{RT}_h$ into $\mathcal{Q}_h$ and is linear/continuous on $\mathcal{RT}_h$, we may apply $\divv$ to the identity above (in the sense of $\mathcal{Q}_h$):
\[
\partial_t(\divv B_h(t)) + \divv(\curl E_h(t)) \;=\; 0 \quad \text{in }\mathcal{Q}_h,\ \text{for a.e.\ }t.
\]
But $\divv(\curl E_h(t))\equiv 0$ elementwise (and hence in $\mathcal{Q}_h$). Therefore
\[
\partial_t(\divv B_h(t)) = 0 \quad \text{in }\mathcal{Q}_h,\ \text{for a.e.\ }t.
\]

\emph{Step 3: Pair with arbitrary $q_h\in\mathcal{Q}_h$.}
For any fixed $q_h\in\mathcal{Q}_h$, the scalar function
\[
g_{q_h}(t):=(\divv B_h(t),q_h)
\]
is absolutely continuous in $t$. Indeed, as in Step 1, above we have $B_h\in H^1(0,T;\mathcal{RT}_h)$. 
Since $\divv:\mathcal{RT}_h\to\mathcal{Q}_h$ is bounded linear and $\mathcal{RT}_h$ is finite
dimensional, write $B_h(t)=\sum_{j=1}^{N_B} b_j(t)\,\rho_j$ with $b_j\in H^1(0,T)$ and
$\{\rho_j\}$ a basis of $\mathcal{RT}_h$. Then
\[
\divv B_h(t)=\sum_{j=1}^{N_B} b_j(t)\,\divv\rho_j \in \mathcal{Q}_h,\qquad
\partial_t(\divv B_h)(t)=\sum_{j=1}^{N_B} b_j'(t)\,\divv\rho_j
= \divv\big(\partial_t B_h(t)\big)
\]
for a.e.\ $t\in(0,T)$. Hence $\divv B_h\in H^1(0,T;\mathcal{Q}_h)$ and
$\partial_t(\divv B_h)=\divv(\partial_t B_h)$ a.e.\ on $(0,T)$. It then follows that 
\[
g_{q_h}(t)=(\divv B_h(t),q_h)\in H^1(0,T)\subset\mathrm{AC}([0,T]),
\]
with, for a.e.\ $t\in(0,T)$,
\[
g_{q_h}'(t)=(\partial_t(\divv B_h(t)),q_h)=(\divv(\partial_t B_h(t)),q_h).
\]
Using the identity $\partial_t B_h+\curl E_h=0$ in $\mathcal{RT}_h$ and
$\divv\circ\curl\equiv 0$, we conclude $g_{q_h}'(t)=0$ a.e.\ on $[0,T)$. Hence
\[
(\divv B_h(t),q_h)=(\divv B_h(0),q_h)\qquad\forall\,q_h\in\mathcal{Q}_h,\ \text{for a.e.\ }t\in[0,T).
\]
This completes the proof.
\end{proof}

\begin{lemma}[Discrete energy identity and bounds]\label{lem:energy}
Define the discrete energy
\[
\mathcal E_h(t):=\tfrac12 \left(\|\sqrt{\varepsilon}E_h(t)\|_{L^2}^2+\|\sqrt{\mu^{-1}}B_h(t)\|_{L^2}^2 \right).
\]
Then for a.e.\ $t$,
\[
\frac{d}{dt}\mathcal E_h(t) + \|\sqrt{\sigma}E_h(t)\|_{L^2}^2 = (f(t),E_h(t)).
\]
Consequently, for any $T>0$,
\[
\sup_{0\le t\le T}\mathcal E_h(t)
\;\le\; e^{T}\mathcal E_h(0)+\frac{e^{T}}{2\varepsilon_{\min}}\int_0^T\|f(s)\|_{L^2}^2\,ds,
\]
and hence
\[
\|E_h\|_{L^\infty(0,T;L^2(\Omega)^3)}+\|B_h\|_{L^\infty(0,T;L^2(\Omega)^3)}\le 
C_T \left( \|f\|_{L^2(0,T;L^2(\Omega)^3)} + \|E_0\|_{L^2(\Omega)^3} + \|B_0\|_{L^2(\Omega)}  \right) ,
\]
with $C_T$ independent of $h$.
\end{lemma}

\begin{proof}
Set $\psi_h=E_h$ in \eqref{eq:SD-amp} and $\phi_h=P_h(\mu^{-1}B_h)$ (recall that $P_h : L^2(\Omega)^3 \rightarrow \RT_h$, see \eqref{eq:projection}) in \eqref{eq:SD-far}, we 
obtain that 
\begin{subequations}\label{eq:SD00}
\begin{align}
(\varepsilon \partial_t E_h,E_h) + (\sigma E_h,E_h) - (\mu^{-1}B_h,\curl E_h) &= (f,E_h) , \label{eq:SD-amp0}
\\
(\partial_t B_h,P_h(\mu^{-1}B_h)) + (\curl E_h,P_h(\mu^{-1}B_h)) &= 0. \label{eq:SD-far0}
\end{align}
\end{subequations}
Notice that $\partial_t B_h \in \RT_h$ and $\curl E_h \in \RT_h$, therefore \eqref{eq:SD-far0} is equivalent to 
\[
(\partial_t B_h,\mu^{-1}B_h) + (\curl E_h,\mu^{-1}B_h) = 0. 
\]
Adding the resulting Faraday and Amp\'ere's laws, we obtain that 
\[
(\varepsilon \partial_t E_h,E_h)+(\sigma E_h,E_h)+(\partial_t B_h,\mu^{-1}B_h)=(f,E_h).
\]
Recognize time derivatives of the quadratic terms to obtain the identity. Apply Young's inequality:
\(
(f,E_h)\le \tfrac{1}{2\varepsilon_{0}}\|f\|^2+\tfrac12 (\varepsilon E_h,E_h)
\),
drop the nonnegative damping term, and apply Gr\"onwall to the inequality
\(\mathcal E_h'(t)\le \mathcal E_h(t)+\frac{1}{2\varepsilon_{0}}\|f\|^2_{L^2(0,T;L^2(\Omega)}\) to obtain the result.
\end{proof}
The above result implies that (up to subsequences)
\begin{equation}\label{eq:hcompact}
\begin{aligned}
E_h &\rightharpoonup E \quad \text{weakly-* in } L^\infty(0,T;L^2(\Omega)^3), \\
B_h &\rightharpoonup B \quad \text{weakly-* in } L^\infty(0,T;L^2(\Omega)^3) .
\end{aligned}
\end{equation}
These limits will be sufficient to pass through the limit in \eqref{eq:SD}, except the time-derivative terms. In the next result, we use these limits to show the convergence of $\divv B_h$. Subsequently, we will analyze the convergence of $\partial_t E_h$ and $\partial_t B_h$. 

\begin{lemma}[Strong convergence of $\divv B_h$ under exact discrete solenoidality]
\label{lem:divBh-strong-zero-upd}
Let the assumptions of Lemma~\ref{lem:disc-gauss-variational} hold. 
If moreover $(\divv B_h(0),q_h)=0$ for all $q_h\in\mathcal Q_h$, then 
\[
\divv B_h \equiv 0 \quad \text{in } L^\infty(0,T;L^2(\Omega)) . 
\]
In particular, $\divv B_h\to 0$ strongly in $L^\infty(0,T;L^2(\Omega))$. 
Furthermore, if $B_h\rightharpoonup B$ weak-* in $L^\infty(0,T;L^2(\Omega)^3)$, then
\[
\divv B_h \rightharpoonup \divv B \quad\text{in }
L^\infty(0,T;L^2(\Omega)), 
\]
and by uniqueness of the weak limit we conclude $\divv B=0$ in 
$L^\infty(0,T;L^2(\Omega))$. 
\end{lemma}

\begin{proof}
For each $t$, since $B_h(t)\in\mathcal{RT}_h$ we have $\divv B_h(t)\in\mathcal Q_h$. By the discrete Gauss law and the divergence-free initialization (cf.~Lemma~\ref{lem:disc-gauss-variational}),
\[
(\divv B_h(t),q_h)=0\qquad\forall\,q_h\in\mathcal Q_h.
\]
Choosing $q_h=\divv B_h(t)$ (admissible because $\divv B_h(t)\in\mathcal Q_h$) yields
$\|\divv B_h(t)\|_{L^2(\Omega)}^2=0$, i.e.\ $\divv B_h(t)\equiv 0$ in $L^2(\Omega)$ for a.e.\ $t$.
Hence $\|\divv B_h\|_{L^\infty(0,T;L^2(\Omega))}=0$, so $\divv B_h\to 0$ strongly in $L^\infty(0,T;L^2(\Omega))$. 

For the identification of the limit, we let $\varphi \in L^1(0,T;C_c^\infty(\Omega))$ and recall that $B_h \rightharpoonup B$ weak-* in $L^\infty(0,T;L^2(\Omega)^3)$, then we have that
\[
\int_0^T (\divv B_h, \varphi) dt 
= -\int_0^T (B_h, \nabla \varphi) dt
\rightarrow -\int_0^T (B, \nabla \varphi) dt = 0 
\]
where in the last equality we have use the uniqueness of limit. From the definition of weak-derivative the required result follows.
\end{proof}

\begin{lemma}[Convergence of $E_h(0)$ in $L^2$]
\label{lem:simple-E0}
Let $E_0\in L^2(\Omega)^3$. With $E_h(0)=Q_h^{\mathrm N}E_0$,
\[
\|E_h(0)-E_0\|_{L^2(\Omega)^3} \xrightarrow[h\to0]{} 0.
\]
\end{lemma}

\begin{proof}
Since $Q_h^{\mathrm N}$ is the $L^2$-orthogonal projector, 
\[
\|E_0-Q_h^{\mathrm N}E_0\|_{L^2(\Omega)^3}=\min_{v_h\in\mathcal N_h^0}\|E_0-v_h\|_{L^2(\Omega)^3}.
\]
The union $\bigcup_h\mathcal N_h^0$ is dense in $L^2(\Omega)^3$, 
hence the best-approximation error vanishes as $h\to0$. Thus $\|E_h(0)-E_0\|_{L^2(\Omega)^3}\to0$.
\end{proof}

\begin{lemma}[Convergence of $B_h(0)$ in $L^2$]
\label{lem:simple-B0} 
Let $B_0\in L^2(\Omega)^3$ with $\divv B_0=0$ in $L^2(\Omega)$.
Assume $\Omega$ is simply connected with connected boundary so that there exists
$A_0\in H_0(\curl;\Omega)$ with $\curl A_0=B_0$.

\smallskip
\noindent\emph{(i) Potential-based initialization.} If $B_h(0):=\curl(R_h A_0)$, then
\[
\|B_h(0)-B_0\|_{L^2(\Omega)^3} \;\le\; \|R_hA_0-A_0\|_{H(\curl;\Omega)} \xrightarrow[h\to0]{} 0.
\]

\smallskip
\noindent\emph{(ii) Constrained $L^2$-projection.} Let $Z_h:=\{v_h\in\RT_h^0:\divv v_h=0\text{ in } \mathcal{Q}_h\}$ and define $B_h(0)\in Z_h$ as the $L^2$-orthogonal projection of $B_0$ onto $Z_h$.
Then
\[
\|B_h(0)-B_0\|_{L^2(\Omega)^3} \xrightarrow[h\to0]{} 0.
\]
\end{lemma}

\begin{proof}
(i) By definition of $B_h(0)$ and $\curl A_0=B_0$,
\[
\|B_h(0)-B_0\|_{L^2(\Omega)^3}=\|\curl(R_hA_0-A_0)\|_{L^2(\Omega)^3}\le \|R_hA_0-A_0\|_{H(\curl;\Omega)}\to0,
\]
using the $H(\curl)$-convergence of $R_h$. 

(ii) Since $B_h(0)$ is the $L^2$-orthogonal projection onto $Z_h$,
\[
\|B_0-B_h(0)\|_{L^2(\Omega)^3}=\min_{z_h\in Z_h}\|B_0-z_h\|_{L^2(\Omega)^3}\le \inf_{w_h\in\curl\Nc_h^0}\|B_0-w_h\|_{L^2(\Omega)^3}.
\]
With $B_0=\curl A_0$ and $w_h=\curl(R_h A_0)\in\curl\Nc_h^0\subset Z_h$, the right-hand side tends to $0$ by part (i). The proof is complete.
\end{proof}

\begin{theorem}[Convergence]\label{lem:RTNconvergence}
Let $E_h\in H^1(0,T;\mathcal N_h^0)$, $B_h\in H^1(0,T;\mathcal{RT}_h)$ solve \eqref{eq:SD}. Then as $h \rightarrow 0$, we have \eqref{eq:hcompact} 
where 
$E,B$ satisfy Definition~\ref{def:weak_soln} (in particular \eqref{eq:MaxWeak}) with $E(0)=E_0$, $B(0)=B_0$, and $\divv B(t)=0$ in $L^2(\Omega)$ a.e. $t \in (0,T)$.
\end{theorem}
\begin{proof}
Let $\Phi\in H(\curl;\Omega)$ be arbitrary. Define $\phi_h = P_h \Phi$ where $P_h$ as is as in \eqref{eq:projection}, then 
\[
\langle \partial_t B_h,\phi_h\rangle = -\,(\curl E_h,P_h \Phi)
= -\,(\curl E_h,\Phi)
= -\,(E_h,\curl \Phi) , 
\]
where in the second equality we have used that definition of $P_h$ and the fact that $\curl E_h \in \mathcal{RT}_h$ from the de Rham sequence. Moreover, in the last equality we have applied the integration-by-parts and have used the boundary conditions on $E_h$. Next, test the equation with $\eta \in C_c^\infty(0,T)$, after integration-by-parts in time, we obtain that 
\[
- \int_0^T ( B_h,\phi_h ) \eta'(t)\, dt
= -\int_0^T (E_h, \rm curl \Phi)\, dt , \quad \forall \Phi \in H(\curl;\Omega)
\]
Now we can take limit as $h\rightarrow 0$ and use that $\phi_h \rightarrow \Phi$ in $L^2$ (cf.~Proposition~\ref{prop:L2proj-RT}) together with \eqref{eq:hcompact} to obtain that, for all $\Phi \in H(\curl;\Omega)$
\[
- \int_0^T ( B,\Phi ) \eta'(t)\, dt
= -\int_0^T (E, \rm curl \Phi)\, dt 
=: \int_0^T \langle L(t), \Phi \rangle_{H(\curl;\Omega)} \eta(t)\, dt .
\]
It is straight forward to see that $L(t) : H(\curl;\Omega) \rightarrow H(\curl;\Omega)^*$ is a bounded linear operator and $\|L(t)\|_{H(\curl;\Omega)^*} \in L^2(0,T)$. Moreover, since $H(\curl;\Omega)$ is dense in $L^2$, therefore we can treat $(B,\Phi)$ as the duality pairing on $H(\curl;\Omega)$. Then from the definition of weak derivative, we have that 
\[
\mathrm{L}(t) = \partial_t B \in L^2(0,T;H(\curl;\Omega)^*) .
\]
Namely, the first equation in \eqref{eq:MaxWeak} holds.

Next, we show that the second equation also holds. From \eqref{eq:SD-amp}, with $\psi_h=R_h\Phi$, where $R_h$ is the Riesz projection (cf.~\ref{eq:def-Rh1}) and $\Phi\in H_0(\curl;\Omega)$, $\eta \in C_c^\infty(0,T)$
\[
\begin{aligned}
\int_0^T \langle \varepsilon\partial_t E_h, R_h \Phi \rangle \eta(t) dt
= \int_0^T \left( -(\sigma E_h,R_h \Phi )+(\mu^{-1}B_h,\curl R_h \Phi )+(f,R_h \Phi ) \right) \eta(t)\, dt.
\end{aligned}
\]
Applying, integration-by-parts on the left-hand-side, we obtain that 
\[
\begin{aligned}
-\int_0^T (\varepsilon E_h, R_h \Phi ) \eta'(t) dt
= \int_0^T \left( -(\sigma E_h,R_h \Phi )+(\mu^{-1}B_h,\curl R_h \Phi )+(f,R_h \Phi ) \right) \eta(t)\, dt.
\end{aligned}
\]
Taking the limits from \eqref{eq:hcompact} and Proposition~\ref{prop:Riesz-Hcurl}, we arrive at 
\[
\begin{aligned}
-\int_0^T (\varepsilon E, \Phi ) \eta'(t) dt
&= \int_0^T \left( -(\sigma E, \Phi )+(\mu^{-1}B,\curl \Phi )+(f, \Phi ) \right) \eta(t)\, dt \\
&=:\int_0^T \langle \mathcal{L}(t), \Phi \rangle_{H_0(\curl;\Omega)} .
\end{aligned}
\]
It is straight-forward to show that for a.e., $t\in (0,T)$, $\mathcal{L} : H_0(\curl;\Omega) \rightarrow H_0(\curl;\Omega)^*$ is bounded and linear and $\|\mathcal{L}\|_{H_0(\curl;\Omega)^*} \in L^2(0,T)$. Thus we have 
\[
\int_0^T \langle \mathcal{L}(t), \Phi \rangle_{H_0(\curl;\Omega)} \eta(t) \, dt 
= -\int_0^T ( E, \Phi  ) \eta'(t)\, dt  
\quad \forall \Phi\in H_0(\curl;\Omega) \, , \forall \eta \in C_c^\infty(0,T) . 
\]
Since $H_0(\curl;\Omega)$ is dense in $L^2(\Omega)$,   
using \cite[Section 7.2]{TRoubicek_2005a} 
we can identify $(E,\Phi)$ as $\langle E, \Phi \rangle_{H_0(\curl;\Omega)}$ a.e. in $t \in (0,T)$. Then from the definition of weak derivative, we deduce that 
\[
 \mathcal{L}(t) = \partial_t E \in L^2(0,T;H_0(\curl;\Omega)^*)
\]
and the second equation in \eqref{eq:MaxWeak} holds. Notice that uniqueness to \eqref{eq:MaxWeak} implies that the entire sequence converges. Finally, the convergence of $\divv B_h$ was shown in Lemma~\ref{lem:divBh-strong-zero-upd}. 
\end{proof}

\appendix

\section{Appendix}
\label{app:appendix}

\begin{proposition}[Global $L^2$-orthogonal projection onto $\RT_h$ and its convergence]
\label{prop:L2proj-RT} 
Let $\RT_h$ denote the Raviart--Thomas space of order $k$ on $\mathcal T_h$, conforming in $H(\divv;\Omega)$.

\smallskip
\noindent\textbf{(a) Construction (global $L^2$-orthogonal projector).}
For $u\in L^2(\Omega)^3$, define $P_h u\in \RT_h$ by
\begin{equation}\label{eq:L2proj-def}
  (u - P_h u,\, v_h)_{L^2(\Omega)} = 0 \qquad \forall\, v_h\in \RT_h.
\end{equation}
Equivalently, if $\{\rho_i\}_{i=1}^M$ is any basis of $\RT_h$, let
$M_{ij}:=(\rho_j,\rho_i)$ and $b_i:=(u,\rho_i)$ and solve $M c=b$; then
$P_h u:=\sum_{j=1}^M c_j\,\rho_j$.

\smallskip
\noindent Then $P_h:L^2(\Omega)^3\to \RT_h$ is a well-defined linear projector with
\[
P_h^2=P_h,\qquad \mathrm{Range}(P_h)=\RT_h,\qquad \|P_h u\|_{L^2}\le \|u\|_{L^2}\quad
\text{and}\quad
\|u-P_h u\|_{L^2}=\min_{v_h\in\RT_h}\|u-v_h\|_{L^2}.
\]

\smallskip
\noindent\textbf{(b) Convergence for $L^2$-data.}
For every $\Phi\in L^2(\Omega)^3$,
\begin{equation}\label{eq:L2-conv}
  \|\,\Phi - P_h\Phi\,\|_{L^2(\Omega)} \xrightarrow[h\to0]{} 0 .
\end{equation}
\end{proposition}

\begin{proof}
\emph{(a)} Pick any basis $\{\rho_i\}$ of $\RT_h$. The matrix $M=[(\rho_j,\rho_i)]$ is
symmetric positive definite (SPD) because $(\cdot,\cdot)$ is an inner product and the basis is
linearly independent, hence the linear system $M c=b$ has a unique solution $c$. This defines a
linear map $P_h$. From \eqref{eq:L2proj-def}, using Pythagoras’ theorem gives, for every $v_h\in\RT_h$,
\[
\begin{aligned}
\|u-v_h\|_{L^2}^2
&=\|u-P_h u\|_{L^2}^2+\|P_h u-v_h\|_{L^2}^2.  
\end{aligned}
\]
Taking the infimum over $v_h$ yields the best-approximation identity. Choosing $v_h = 0$ gives the contractivity $\|P_h u\|_{L^2}\le \|u\|_{L^2}$. 

To prove the range, we first notice that by definition $P_h u \in \RT_h$ for every $u \in L^2$. Hence $Range(P_h) \subset \RT_h$. Next let $v_h \in \RT_h$. Applying orthogonality condition with $u = v_h$, we obtain that 
\[
(v_h - P_h v_h, w_h) = 0 \quad \forall w_h \in \RT_h .
\]
Since $v_h - P_h v_h \in \RT_h$, choose $w_h = v_h-P_h v_h$ to get 
\[
\|v_h - P_h v_h\|^2_{L^2} = 0 \quad \implies \quad P_h v_h = v_h .
\]
Thus every $v_h \in \RT_h$ is in the range of $P_h$, so $\RT_h \subset Range(P_h)$. Finally, the idempotence follows immediately.

\smallskip
\emph{(b)} We prove density of $\bigcup_h \RT_h$ in $L^2(\Omega)^3$, which combined with
best approximation gives \eqref{eq:L2-conv}. Let $\varepsilon>0$ and choose
$\phi\in C_c^\infty(\Omega)^3$ with $\|\Phi-\phi\|_{L^2}<\varepsilon$ (density of $C_c^\infty$ in
$L^2$). For $\phi\in H^1(\Omega)^3$ there exist standard Raviart--Thomas interpolants (or smoothed
quasi-interpolants, \cite{HAntil_SBartels_AKaltenbach_RKhandelwal_2024a} and \cite[Thms. 16.4, 16.6]{AErn_JLGuermond_2020a}) $I_h^{\mathrm{RT}}\phi\in \RT_h$ such that
\begin{equation}\label{eq:RT-approx}
  \|\phi - I_h^{\mathrm{RT}}\phi\|_{L^2(\Omega)} \;\le\; C\, h\,\|\phi\|_{H^{1}(\Omega)}.
\end{equation}
Using the optimality of $P_h$ and stability $\|P_h\|\le1$,
\[
\|\Phi - P_h\Phi\|_{L^2}
\le \|\Phi - \phi\|_{L^2} + \|\phi - P_h\phi\|_{L^2} + \|P_h(\phi-\Phi)\|_{L^2}
\le 2\varepsilon + \|\phi - I_h^{\mathrm{RT}}\phi\|_{L^2}.
\]
By \eqref{eq:RT-approx}, the last term tends to $0$ as $h\to0$. Since $\varepsilon>0$ is arbitrary,
\(\|\Phi-P_h\Phi\|_{L^2}\to0\). This proves \eqref{eq:L2-conv}. 
\end{proof}

\begin{proposition}[Riesz projection on $H_0(\curl)$: stability and convergence]
\label{prop:Riesz-Hcurl} 
Define, for each $u\in H_0(\curl;\Omega)$, the element $R_h u\in \mathcal N_h^0$ by
\begin{equation}\label{eq:def-Rh}
  a(R_h u, v_h) \;=\; a(u, v_h)\qquad\forall\,v_h\in \mathcal N_h^0,
\end{equation}
where $a(u,v):=(u,v)+(\curl u,\curl v)$.
Then the following hold.

\smallskip\noindent{\bf (i) Well-posedness and linearity.}
For each $u$, there exists a unique $R_h u\in\mathcal N_h^0$ solving \eqref{eq:def-Rh}; the
map $R_h:H_0(\curl;\Omega)\to \mathcal N_h^0$ is linear.

\smallskip\noindent{\bf (ii) Projection and Galerkin orthogonality.}
$R_h$ is a projection: $R_h|_{\mathcal N_h^0}=\mathrm{Id}$.
Moreover,
\begin{equation}\label{eq:gal-orth}
  a(u - R_h u,\, v_h)=0\qquad \forall\,v_h\in \mathcal N_h^0.
\end{equation}

\smallskip\noindent{\bf (iii) Stability (contractivity) in $H(\curl)$.}
For all $u\in H_0(\curl;\Omega)$,
\begin{equation}\label{eq:contract}
  \|R_h u\|_{H(\curl)} \;\le\; \|u\|_{H(\curl)},
\qquad
  \|u-R_h u\|_{H(\curl)} \;=\; \min_{v_h\in\mathcal N_h^0}\|u-v_h\|_{H(\curl)}.
\end{equation}

\smallskip\noindent{\bf (iv) Convergence (general).} 
For every $u\in H_0(\curl)$,
\begin{equation}\label{eq:conv-general}
  \|u-R_h u\|_{H(\curl)} \xrightarrow[h\to0]{} 0,
\quad\text{hence}\quad
  \|u-R_h u\|_{L^2} \to 0,\ \ \|\curl(u-R_h u)\|_{L^2}\to 0 .
\end{equation}
\end{proposition}

\begin{proof}
\emph{(i) Well-posedness.}
Endow $H_0(\curl;\Omega)$ with the inner product $a(\cdot,\cdot)$ and induced norm
$\|v\|_{H(\curl;\Omega)}^2=a(v,v)$. On $\mathcal N_h^0$, $a(\cdot,\cdot)$ is symmetric and coercive:
$a(v_h,v_h)=\|v_h\|_{H(\curl;\Omega)}^2>0$ for $v_h\neq0$. Thus, by the Riesz representation theorem in
the finite-dimensional subspace $\mathcal N_h^0$, for each continuous linear functional
$v_h\mapsto a(u,v_h)$ there exists a unique $R_h u\in\mathcal N_h^0$ solving \eqref{eq:def-Rh}.
Linearity follows from linearity of \eqref{eq:def-Rh} in $u$.

\emph{(ii) Projection and orthogonality.}
If $u\in\mathcal N_h^0$, then taking $v_h=R_h u-u\in\mathcal N_h^0$ in \eqref{eq:def-Rh} gives
$a(R_h u - u,\, R_h u - u)=0$, hence $R_h u=u$. Subtracting \eqref{eq:def-Rh} from itself yields
$a(u-R_h u,v_h)=0$ for all $v_h\in\mathcal N_h^0$.

\emph{(iii) Stability and best approximation.}
Since $R_h$ is the orthogonal projector onto $\mathcal N_h^0$ in the Hilbert space
$\big(H_0(\curl;\Omega),a(\cdot,\cdot)\big)$, Pythagoras’ theorem gives, for every $v_h\in\mathcal N_h^0$,
\[
\|u-v_h\|_{H(\curl;\Omega)}^2
=\|u-R_h u\|_{H(\curl;\Omega)}^2+\|R_h u-v_h\|_{H(\curl;\Omega)}^2.
\]
Taking the infimum over $v_h$ yields the best-approximation identity in \eqref{eq:contract}.
Choosing $v_h=0$ shows
$\|u\|_{H(\curl;\Omega)}^2=\|u-R_h u\|_{H(\curl;\Omega)}^2+\|R_h u\|_{H(\curl;\Omega)}^2\ge \|R_h u\|_{H(\curl;\Omega)}^2$,
hence contractivity.

\emph{(iv) Convergence.}
Let $u\in H_0(\curl)$ and $\varepsilon>0$. By density of $\bigcup_h\mathcal N_h^0$ in $H_0(\curl;\Omega)$
there exists $v_h^\varepsilon\in\mathcal N_h^0$ with
$\|u-v_h^\varepsilon\|_{H(\curl;\Omega)}<\varepsilon$. By best approximation,
\[
\|u-R_h u\|_{H(\curl;\Omega)} \;=\; \min_{w_h\in\mathcal N_h^0}\|u-w_h\|_{H(\curl;\Omega)}
\;\le\; \|u-v_h^\varepsilon\|_{H(\curl;\Omega)} \;<\; \varepsilon.
\]
As $\varepsilon$ was arbitrary, $\|u-R_h u\|_{H(\curl;\Omega)}\to0$. The two component convergences in
\eqref{eq:conv-general} follow since the $H(\curl;\Omega)$-norm is
$\|u-R_h u\|_{H(\curl)}^2=\|u-R_h u\|_{L^2}^2+\|\curl(u-R_h u)\|_{L^2}^2$.
\end{proof}

\begin{proposition}[Constrained $L^2$-projection onto discrete divergence-free RT fields]
\label{prop:constrained-proj}
Let $\mathcal{RT}_h\subset H(\divv;\Omega)$ be a Raviart--Thomas space (order $k\ge0$) on a shape-regular mesh $\mathcal T_h$, and let
$\mathcal Q_h:=\{q_h\in L^2(\Omega): q_h|_K\in\mathbb P_k(K)\ \forall K\in\mathcal T_h\}$.
Define the discrete divergence-free subspace
\[
Z_h := \ker(\divv|_{\mathcal{RT}_h})=\{v_h\in\mathcal{RT}_h:\ \divv v_h=0\ \text{in }\mathcal Q_h\}.
\]
For any $B_0\in L^2(\Omega)^3$ there exists a unique $B_h^0\in Z_h$ such that
\begin{equation}\label{eq:orth-onto-Zh}
(B_h^0, z_h) = (B_0, z_h) \qquad \forall\, z_h\in Z_h.
\end{equation}
Equivalently, $(B_h^0,p_h)\in \mathcal{RT}_h\times\mathcal Q_h$ solves the mixed system
\begin{equation}\label{eq:mixed-proj}
\begin{aligned}
(B_h^0, v_h) + (p_h,\divv v_h) &= (B_0, v_h)&&\forall\, v_h\in \mathcal{RT}_h,\\
(\divv B_h^0, q_h) &= 0 &&\forall\, q_h\in \mathcal Q_h.
\end{aligned}
\end{equation}
Moreover, $B_h^0$ is the unique minimizer of the constrained problem
\[
\min\Big\{ \tfrac12\|v_h-B_0\|_{L^2(\Omega)^3}^2:\ v_h\in\mathcal{RT}_h,\ \divv v_h=0\ \text{in }\mathcal Q_h\Big\}.
\]
\end{proposition}

\begin{proof}
\emph{Step 1: Equivalence ``orthogonality $\Leftrightarrow$ minimizer''.}
Let $Z_h$ be a finite-dimensional subspace of the Hilbert space $L^2(\Omega)^3$.
The (unconstrained) best approximation problem $\min_{z_h\in Z_h}\|z_h-B_0\|^2$ has a unique solution characterized by the $L^2$-orthogonality
\((B_0-B_h^0,z_h)=0\) for all $z_h\in Z_h\), i.e.\ \eqref{eq:orth-onto-Zh}.
This is the standard projection theorem in Hilbert spaces.

\emph{Step 2: Equivalence with the mixed (saddle-point) system.}
Consider the constrained minimization
\(\min_{v_h\in\mathcal{RT}_h}\tfrac12\|v_h-B_0\|^2\) subject to \(\divv v_h=0\in\mathcal Q_h\).
Introduce a Lagrange multiplier \(p_h\in\mathcal Q_h\) and the Lagrangian
\[
\mathcal L(v_h,p_h) := \tfrac12 (v_h,v_h) - (B_0,v_h) + (p_h,\divv v_h).
\]
Stationarity w.r.t.\ $v_h$ and $p_h$ yields \eqref{eq:mixed-proj}. Conversely, if
\((B_h^0,p_h)\) solves \eqref{eq:mixed-proj}, then for any $z_h\in Z_h$ the first line with $v_h=z_h$ gives
\((B_h^0-B_0,z_h)=0\), i.e.\ \eqref{eq:orth-onto-Zh}, so $B_h^0\in Z_h$ is the (unique) orthogonal projection.

\emph{Step 3: Existence and uniqueness of \eqref{eq:mixed-proj}.}
Let $a(v_h,w_h):=(v_h,w_h)$ on $\mathcal{RT}_h$ and $b(v_h,q_h):=(\divv v_h,q_h)$.
The pair $(\mathcal{RT}_h,\mathcal Q_h)$ satisfies the uniform Babu\v{s}ka--Brezzi conditions:
(i) $a(\cdot,\cdot)$ is coercive on $\ker b = Z_h$ because $a(z_h,z_h)=\|z_h\|_{L^2}^2$;
(ii) the discrete inf--sup holds:
\[
\inf_{0\neq q_h\in \mathcal Q_h}\ \sup_{0\neq v_h\in\mathcal{RT}_h}
\frac{(\divv v_h,q_h)}{\|v_h\|_{L^2}\,\|q_h\|_{L^2}} \ \ge\ \beta>0,
\]
with $\beta$ independent of $h$ (standard for RT spaces; e.g.\ via a Fortin operator that commutes with $\divv$).
Thus \eqref{eq:mixed-proj} is well-posed, giving a unique pair $(B_h^0,p_h)$ and hence a unique $B_h^0\in Z_h$.
\end{proof}

\section*{Acknowledgment}
First and foremost, the author thanks Alex Kaltenbach for carefully reading the manuscript and offering helpful suggestions, especially regarding uniqueness of the solution to the continuous problem. The author is grateful to Irwin Yousept for providing key references on Maxwell’s equations \cite{GDuvaut_JLLions_1976a,RLeis_1997a,AKirsch_ARieder_2016a,IYousept_2020a,MFabrizio_MMorro_2003a} and for many discussions; in particular, he pointed out that it suffices to assume $\mu\in L^\infty$. The author also thanks Keegan Kirk for proofreading the manuscript. Finally, the author is indebted to Denis Ridzal and to Sandia National Laboratories’ LDRD project \emph{Active Circuits for EMI Resilience in Contested Environments} for valuable input on the technical direction.

\bibliographystyle{plain}
\bibliography{refs}

\end{document}